\documentclass[12pt]{amsart}
\usepackage{amsmath,amsfonts,amsthm,amssymb,mathrsfs,amsxtra,amscd,latexsym}
\usepackage[hmargin=2cm,vmargin=3cm]{geometry}
\usepackage[all]{xy}
\newtheorem{thm}{Theorem}[section]

\newtheorem{cor}[thm]{Corollary}
\newtheorem{lem}[thm]{Lemma}

\newtheorem{dfn}[thm]{Definition}

\newtheorem{rmk}[thm]{Remark}

\theoremstyle{definition}
\theoremstyle{remark}
\numberwithin{equation}{section}

\newcommand{\sm}{\left(\begin{smallmatrix}}
\newcommand{\esm}{\end{smallmatrix}\right)}
\newcommand{\mat}{\left(\begin{matrix}}
\newcommand{\emat}{\end{matrix}\right)}

\newcommand{\mbf}{\mathbf}

\def\CC{\mathbb{C}}
\def\HH{\mathbb{H}}
\def\NN{\mathbb{N}}
\def\QQ{\mathbb{Q}}
\def\RR{\mathbb{R}}
\def\ZZ{\mathbb{Z}}

\def\det{\mathrm{det}}

\def\m{\mathrm{mod}}

\def\re{\mathrm{Re}}
\def\reg{\mathrm{reg}}

\def\Tr{\mathrm{Tr}}

\def\PSL{\mathrm{PSL}}
\def\SL{\mathrm{SL}}

\begin{document}


\title{limits of traces of singular moduli}


\author{Dohoon Choi}
\author{Subong Lim}

\address{Department of Mathematics, Korea University, 145 Anam-ro, Seongbuk-gu, Seoul 02841, Republic of Korea}
\email{dohoonchoi@korea.ac.kr}	

\address{Department of Mathematics Education, Sungkyunkwan University, 25-2, Sungkyunkwan-ro, Jongno-gu, Seoul 03063, Republic of Korea}
\email{subong@skku.edu}

\thanks{
The authors were supported by Samsung Science and Technology Foundation
under Project SSTF-BA1301-11.
}

\begin{abstract}
Let $f$ and $g$ be weakly holomorphic modular functions on $\Gamma_0(N)$ with the trivial character. For an integer $d$, let $\Tr_d(f)$ denote the modular trace of $f$ of index $d$. Let $r$ be a rational number equivalent to $i\infty$ under the action of $\Gamma_0(4N)$.
In this paper, we prove that, when $z$ goes radially to $r$, the limit $Q_{\hat{H}(f)}(r)$ of the sum $H(f)(z) = \sum_{d>0}\Tr_d(f)e^{2\pi idz}$ is a special value of a regularized twisted $L$-function defined by $\Tr_d(f)$ for $d\leq0$. It is proved that the regularized $L$-function is meromorphic  on $\mathbb{C}$ and satisfies a certain functional equation.  Finally, under the assumption that $N$ is square free, we prove that if $Q_{\hat{H}(f)}(r)=Q_{\hat{H}(g)}(r)$ for all $r$ equivalent to $i \infty$ under the action of $\Gamma_0(4N)$, then $\Tr_d(f)=\Tr_d(g)$ for all integers $d$.
\end{abstract}

\maketitle

\section{Introduction}\label{section1}

The   $j$-invariant function is defined by
\[
j(z) = q^{-1} + 744 + 196884q +\cdots,
\]
for $z$ in the complex upper half plane $\HH$, where $q:= e^{2\pi iz}$.
Let $j_1 := j - 744$, which is the normalized Hauptmodul for the group $\Gamma = \PSL_2(\ZZ)$.
For an integer $d$ such that $d\equiv 0$ or  $1\ (\m\ 4)$, let $\mathcal{Q}_d$ be the set of integral binary quadratic forms
\[
Q(x,y) = ax^2 + bxy + cy^2
\]
with discriminant $d = b^2 - 4ac$. Let us assume that $Q$ is positive definite if $d<0$.
The group $\Gamma$ acts on $\mathcal{Q}_d$ by
\[
\left(Q \circ \sm \alpha& \beta\\ \gamma& \delta\esm\right)(x,y) = (\alpha x+\beta y, \gamma x+\delta y).
\]
Let $\Gamma_Q$ denote the  stabilizer of $Q$ in $\Gamma$.

For $d<0$ and $Q(x,y) = ax^2 + bxy + cy^2 \in\mathcal{Q}_d$, the root $z_Q := \frac{-b+\sqrt{d}}{2a}$ defines a CM point in $\HH$.
The values of the function $j_1$ at  CM points are called {\it singular moduli}.
The trace of singular moduli of index $d$ is defined by
\begin{equation*} \label{negativetrace}
\Tr_d(j_1) := \sum_{Q\in \Gamma\setminus \mathcal{Q}_d} \frac{1}{|\Gamma_Q|} j_1(z_Q).
\end{equation*}
Zagier \cite{Zag} proved that
 the generating series of $\Tr_d(j_1)$, defined by
\begin{equation} \label{g1}
g_1(z) := -q^{-1} + 2 + \sum_{d<0} \Tr_d(j_1)q^{|d|} = -q^{-1} + 2 - 248q^3 + 492q^4 - 4119q^7 + 7256q^8+\cdots,
\end{equation}
 is a weakly holomorphic modular form of weight $\frac32$ and the  multiplier system $\chi_{\theta}^3$ on $\Gamma_0(4)$. Here, $\theta(z) := 1 + 2\sum_{n=1}^\infty q^{n^2}$ is the theta function, and $\chi^{}_\theta$ is the theta multiplier system of weight $\frac12$ on $\Gamma_0(4)$ defined by
 $$\chi^{}_{\theta}(\gamma) := \frac{\theta(\gamma z)}{(cz+d)^{\frac12}\theta(z)}$$
 for $\gamma = \sm a&b\\c&d\esm\in \Gamma_0(4)$.
 In this paper, we use the convention:
$z = |z|e^{i\mathrm{arg}(z)},\ -\pi< \mathrm{arg}(z) \leq \pi$ as in \cite{BF}.
Bruinier and Funke \cite{BF2} extended the result of Zagier  to weakly holomorphic modular functions on modular curves of arbitrary genus.

For $d>0$, the modular trace $\Tr_d(j_1)$ of index $d$ is defined by the sum of cycle integrals of $j_1$.
Cycle integrals were used to define the Shintani lifting in \cite{Shi}, and Kaneko \cite{Kan} used them to compute values of $j(\tau)$ at real qudratics.

Let us assume that $d>0$, and that $Q(x, y) = ax^2 + bxy + cy^2$ is in $\mathcal{Q}_d$. If $a, b$, and $c$ have no non-trivial common divisor, then, for a non-square integer $d$, we let $t$ and $u$ be the smallest half-integral solutions of the Pell equation $t^2 - du^2 = 1$, and
\[
g_Q := \sm t+bu& 2cu\\ -2au& t-bu\esm \in \SL_2(\ZZ).
\]
Let $\delta$ be the greatest common divisor of $a, b$, and $c$, and $Q':=Q/\delta$. Then $c(Q)$ denotes any rectifiable curve in $\HH$ from $z$ to $g_{Q'}z$, where $z$ is any point on $\HH$.
For a non-square integer $d$, the modular trace $\Tr_d(j_1)$ of index $d$ is defined by
\[
\Tr_d(j_1) := \frac{1}{2\pi}\sum_{Q\in\Gamma\setminus \mathcal{Q}_d} \int_{c(Q)} j_1(\tau)\frac{d\tau}{Q(\tau,1)}
\]
(see \cite[Section 1]{DIT3} for the definition of $\Tr_d(j_1)$  when $d$ is a square of an integer).
In \cite{DIT3}, Duke, Imamo$\bar{\mathrm{g}}$lu, and T\'oth proved that the generating series of $\Tr_d(j_1)$ for $d>0$, defined by
\begin{equation*}
F(z) :=\sum_{d>0} \Tr_d(j_1)q^d
\end{equation*}
is a mock modular form of weight $\frac12$ and multiplier system $\chi^{}_\theta$ on $\Gamma_0(4)$ such that its shadow is $-2g_1$ in (\ref{g1}), i.e., $\xi_{\frac12}\left(\hat{F}\right) = -2g_1$, where $\xi_{k} := 2iy^k \overline{\frac{\partial}{\partial\bar{z}}}$, $y$ is the imaginary part of $z$, and $\hat{F}$ is a harmonic weak Maass form of weight $\frac12$ on $\Gamma_0(4)$ obtained as a completion of $F$.
Recently, Bruinier, Funke, and  Imamo$\bar{\mathrm{g}}$lu \cite{BFI} generalized the result of  Duke, Imamo$\bar{\mathrm{g}}$lu, and T\'oth to weakly holomorphic modular functions on modular curves of arbitrary genus. They also gave a geometric interpretation of $\Tr_d$ for square integers $d$.

In this paper, we investigate connections between $\Tr_d(f)\ (d>0)$ and $\Tr_d(f)\ (d<0)$ for a weakly holomorphic modular function $f$ by considering a certain asymptotic behavior of twisted sums of $\Tr_d(f)$ over $d>0$ and $d<0$, respectively. For this purpose, we define a regularized twisted $L$-function of a weakly holomorphic modular form, which has a meromorphic continuation on $\mathbb{C}$ and satisfies a certain functional equation.

To state our main theorem, we introduce more notation.
Let $N$ be a positive integer.
For $k\in \frac{1}{2}\mathbb{Z}$, let
$M^!_{k}(\Gamma_0(4N))$ be the space of weakly holomorphic modular forms of weight $k$ and multiplier $\chi_{\theta}^{2k}$ on $\Gamma_0(4N)$, where a weakly holomorphic modular form is a meromorphic modular form which has poles only at cusps.
For a discrete subgroup $\Gamma$ of $\SL_2(\RR)$,
let $M^!_{0}(\Gamma)$ be the space of weakly holomorphic modular functions on $\Gamma$ with the trivial character.
Let $\mathcal{Q}_{d,N}$ be  the set of quadratic forms $ax^2 + bxy +cy^2 \in\mathcal{Q}_d$ such that $a\equiv 0\ (\m\ N)$.
The group $\Gamma_0(N)$ acts on $\mathcal{Q}_{d,N}$ with finitely many orbits.

Let $f\in M^!_{0}(\Gamma_0(N))$ with a Fourier expansion of the form
\begin{equation} \label{Fourierinfinite}
f(z) = \sum_{n\gg-\infty} a(n)e^{2\pi inz}
\end{equation}
and assume that its constant coefficients  vanish at all cusps of $\Gamma_0(N)$.
For a nonzero integer $d$, the {\it modular trace} $\Tr_{d}(f)$ of $f$ of index $d$ is defined by
\begin{equation*}
\Tr_d(f) :=
\begin{cases}
\sum_{Q\in \Gamma_0(N)\setminus \mathcal{Q}_{d,N}} \frac{1}{|\Gamma_0(N)_Q|} f(z_Q) & \text{if $d<0$},\\
\frac1{2\pi} \sum_{Q\in\Gamma_0(N)\setminus\mathcal{Q}_{d,N}} \int^{\reg}_{\Gamma_0(N)_Q\setminus c_Q} f(z)\frac{dz}{Q(z,1)}& \text{if $d>0$},\\
-\frac{1}{2\pi}\int^{\mathrm{reg}}_{\Gamma_0(N)\setminus \HH} f(z)\frac{dxdy}{y^2} & \text{if $d = 0$},
\end{cases}
\end{equation*}
where $c_Q$ is a properly oriented geodesic in $\HH$ connecting two roots of $Q(z,1)$ and $\int^{\reg}_{\Gamma_0(N)_Q\setminus c_Q}$ and $\int^{\reg}_{\Gamma_0(N)\setminus\HH}$ are regularized integrals, as defined by Bruinier, Funke, and  Imamo$\bar{\mathrm{g}}$lu \cite{BFI} (for their precise definitions, see Section \ref{tracesingular}).

It was proved in \cite{BFI} that the generating series of  $\Tr_d(f)$ for $d>0$, defined by
\begin{equation} \label{mockpart}
H(f)(z) := \sum_{d>0} \Tr_d(f)e^{2\pi idz},
\end{equation}
is a mock modular form of weight $\frac12$ on $\Gamma_0(4N)$.
Furthermore, the shadow is $-2W(f)$, where $W(f)$ is the generating series of $\Tr_d(f)$ with $d\geq0$.
For example, if $f$ is a weakly holomorphic modular function on $\Gamma_0^*(N)$ for a square-free integer $N$, then $W(f)$ is given by
\begin{equation} \label{shadowW}
W(f)(z) = -\sum_{d>0}\sum_{n<0}d\overline{a(dn)}q^{-d^2} + \frac12\overline{\Tr_0(f)} + \sum_{d<0}\overline{\Tr_d(f)}q^{-d}.
\end{equation}
Here, $\Gamma_0^*(N)$ denotes the extension of $\Gamma_0(N)$ by the Atkin-Lehner involutions $W_p = \sm 0&-1\\p&0\esm$ for all primes $p | N$. For a weakly holomorphic modular function $f$ on $\Gamma_0(N)$ with an arbitrary positive integer $N$, we give an explicit expression for $W(f)$ in Section \ref{workBFI}.

For $k\in\frac12\ZZ$, let $g$ be a weakly holomorphic modular form in $M^!_{k}(\Gamma_0(4N))$. Let $\Gamma(s,x)$ be the incomplete gamma function as in (6.5.3) of \cite{AS} (see also Section 3 of \cite{BF}).
Assume that, for each $\gamma = \sm a&b\\c&d\esm\in\SL_2(\ZZ)$, the function $(cz+d)^{-k}g(\gamma z)$ has a Fourier expansion of the form
\begin{equation} \label{fourierq}
 \sum_{n\gg-\infty} b_\gamma(n)e^{2\pi i(n+\kappa_\gamma)z/\lambda_\gamma},
\end{equation}
where $\kappa_\gamma\in[0,1)$ and $\lambda_\gamma$ is a positive integer.
If $\gamma = I$, then we drop $\gamma$ from the notation.
For $r\in\QQ$, we define a {\it regularized twisted $L$-function of $g$ associated with $r$} by
\begin{eqnarray} \label{generalcuspL}
\nonumber L^{\reg}_r(g,s) &:=& \frac{1}{\Gamma(s)}\sum_{n\gg-\infty\atop n\neq0} \frac{b(n)e^{2\pi inr}}{n^s} \Gamma(s, 2\pi n) \\
&&+ \frac{i^{k}}{c^k\Gamma(s)} (2\pi)^{2s-k} (c^2)^{k-s}
\sum_{n\gg-\infty\atop n+\kappa_{\gamma}\neq0} \frac{b_\gamma(n)e^{2\pi i(n+\kappa_\gamma)(-\frac{d}{c})/\lambda_\gamma}}{\left(\frac{n+\kappa_\gamma}{\lambda_\gamma}\right)^{k-s}} \Gamma\left(k-s, \frac{2\pi(n+\kappa_\gamma)}{c^2\lambda_\gamma}\right)\\
\nonumber&& -\frac{(2\pi)^{s} b(0)}{s\Gamma(s)}- \delta_{\kappa_\gamma,0} \frac{i^k (2\pi)^s b_\gamma(0)}{ c^k(k-s)\Gamma(s)},
\end{eqnarray}
where $\delta_{\kappa_\gamma,0}$ is the Kronecker delta, and $\gamma = \sm a&b\\c&d\esm\in\SL_2(\ZZ)$  such that
$r= \gamma(i\infty)$.
The definition of $L^{\reg}_r(g,s)$ is independent of the choice of $\gamma$ (see Lemma \ref{Lseries} (2)).

To state the first theorem, we introduce some notations.
For odd $d$, let $\varepsilon_d$ be defined by
\begin{equation*}
\varepsilon_d  :=
\begin{cases}
1 & \text{if $d\equiv 1\ (\mathrm{mod}\ 4)$},\\
i & \text{if $d\equiv 3\ (\mathrm{mod}\ 4)$}.
\end{cases}
\end{equation*}
If $d$ is an odd prime, then let $\left( \frac cd \right)$ be the usual Legendre symbol.
For positive odd $d$, define $\left( \frac cd \right)$ by multiplicativity.
For negative odd $d$, we let
\begin{equation*}
\left( \frac cd \right) =
\begin{cases}
\left( \frac {c}{|d|}\right) & \text{if $d<0$ and $c>0$},\\
-\left( \frac {c}{|d|}\right) & \text{if $d<0$ and $c<0$}.
\end{cases}
\end{equation*}
We also let $\left( \frac{0}{\pm1}\right) = 1$.
 The following theorem shows that, as $z$ approaches a rational number, the limit of the sum $\sum_{d>0}\Tr_d(f)e^{2\pi idz}$ is a special value of a regularized twisted  $L$-function defined by $\Tr_d(f)$ for $d\leq0$.

\begin{thm} \label{trace}
Let $N$ be a positive integer. Suppose that $f\in M^!_{0}(\Gamma_0(N))$ with Fourier expansion as in (\ref{Fourierinfinite}),
and that its constant coefficients vanish at all cusps of $\Gamma_0(N)$.
If $r=\frac ac$ is a rational number such that $c\equiv 0\ (\m\ 4N)$, $\mathrm{gcd}(a,c) = 1$ and $c>0$,
then
\begin{equation} \label{radial}
\lim_{t \to 0} \left( \sum_{d>0} \Tr_d(f)e^{2\pi id(r+it)}+ \frac{2c_r\sqrt{1+ct}}{ct} \right) = - \overline{L^{\reg}_r\left(W(f), \frac12\right)}+c_r,
\end{equation}
where $$c_r:=\frac{e^{i\pi/4}\varepsilon_a^{3}}{\sqrt{2c}}\left(\frac{c}{a}\right)\Tr_0(f).$$
\end{thm}

\begin{rmk} Let $g\in M_{k}^{!}(\Gamma_0(4N))$ for $k\in\frac12\ZZ$. The regularized $L$-functions of weakly holomorphic modular forms satisfy the following properties  (for the details, see Proposition \ref{Lseries} in Section \ref{Regularized $L$-function}).
\begin{enumerate}
\item
The series $L^{\reg}_r(g,s)$ defined in (\ref{generalcuspL}) converges absolutely  on $\mathbb{C}$ except $s=0$ and $s=k$.
\item
If $g$ is a cusp form, then the usual $L$-function $L(g,s)$ of $g$ is the same as $L^{\reg}_0(g,s)$.
\end{enumerate}
\end{rmk}

\begin{rmk}
The limit in (\ref{radial}) is the radial limit that was used by Ramanujan in explaining what he meant by a mock theta function, in his last letter to Hardy (for example, see \cite{AH}).
Ramanujan claimed that, for a mock theta function $f(q)$ and a root of unity $\zeta$, there is a modular form
$
\theta_{f,\zeta}(q)
$
such that the difference
\begin{equation} \label{difference}
f(q) - \theta_{f,\zeta}(q)
\end{equation}
  is bounded as $q\to\zeta$ radially.
Watson \cite{Wat} was the first to prove Ramanujan's claim concerning the mock theta function
\[
1 + \frac{q}{(1+q)^2} + \frac{q^4}{(1+q)^2(1+q^2)^2} + \cdots.
\]
Folsom, Ono, and Rhoades \cite{FOR} gave a new proof of this result, and obtained a simple closed formula for the suggested $O(1)$ constants as values of a ``quantum'' $q$-hypergeometric series.
The authors and Rhoades \cite{CLR} 
computed the radial limit of (\ref{difference}) by using special values of the twisted $L$-functions of their shadows.
\end{rmk}

Duke, Imamo$\bar{\mathrm{g}}$lu, and T\'oth \cite{DIT0, DIT3} studied period functions of modular integrals whose Fourier coefficients are given by cycle integrals of certain weakly holomorphic modular forms.
Note that $H(f)$ defined in (\ref{mockpart}) is a mock modular form of weight $\frac12$.
The following corollary describes
 the limit behavior of period functions of  the mock modular form $H(f)$.

\begin{cor} \label{period}
Let
$
\psi(\hat{H}(f))_\gamma := H(f) - H(f)|_{\frac12}\gamma
$
for $\gamma\in\Gamma_0(4N)$.
Let $\gamma = \sm a&b\\c&d\esm\in \Gamma_0(4N)$ with $r = \gamma(\infty)\in\QQ$ and $c>0$.
Then, we have
\begin{equation}  \label{psiHflimit}
\lim_{t \to 0} \left( \psi(\hat{H}(f))_{\gamma^{-1}}(r+it) + \frac{2c_r\sqrt{1+ct}}{ct} \right) = - \overline{L^{\reg}_r\left(W(f), \frac12\right)}+c_r.
\end{equation}
\end{cor}

For $r\in\QQ$ equivalent to $i\infty$ under the action of $\Gamma_0(4N)$, let
\begin{equation} \label{qhf}
Q_{\hat{H}(f)}(r): =  \lim_{t \to 0} \left( \sum_{d>0} \Tr_d(f)e^{2\pi id(r+it)}+ \frac{2c_r\sqrt{1+ct}}{ct} \right).
\end{equation}
In the following theorem, we prove that, for two weakly holomorphic modular functions $f$ and $g$ on $\Gamma_0^*(N)$ with a square-free integer $N$,  if $Q_{\hat{H}(f)}(r)$ and $Q_{\hat{H}(g)}(r)$ have the same values at rational numbers $r$ equivalent to $i\infty$ under the action of $\Gamma_0(4N)$, then $\mathrm{Tr}_d(f) = \mathrm{Tr}_{d}(g)$ for each index $d$.

\begin{thm} \label{determine}
Let $N$ be a square-free integer and  $f,g\in M^!_{0}(\Gamma_0^*(N))$ such that  their
 constant terms vanish in Fourier expansions at $i\infty$.
Suppose that the identity
\[
Q_{\hat{H}(f)}(r) = Q_{\hat{H}(g)}(r)
\]
holds for all rational numbers $r$ equivalent to $i\infty$ under the action of $\Gamma_0(4N)$.
Then, for each integer $d$,
\[
\mathrm{Tr}_d(f) = \mathrm{Tr}_{d}(g).
\]
\end{thm}

The remainder of this paper is organized as follows. In Section \ref{tracesingular}, we review the basic notions of modular traces and their modularity which were proved by Duke, Imamo$\bar{\mathrm{g}}$lu, T\'oth \cite{DIT3} and Bruinier, Funke, Imamo$\bar{\mathrm{g}}$lu \cite{BFI}.
In Section \ref{regular}, we define regularized Eichler integrals for weakly holomorphic modular forms, and then we prove that the regularized Eichler integrals are well defined and that, for a harmonic weak Maass form $f$, the nonholomorphic part of $f$ is the same as the regularized Eichler integral of its shadow. In Section \ref{Regularized $L$-function}, we define a regularized $L$-function for a weakly holomorphic
modular form, and we prove that the regularized $L$-function is meromorphic on $\mathbb{C}$ and satisfies a certain functional equation.
In Section \ref{proof1}, we prove Theorem \ref{trace} and Corollary \ref{period}.
In Section \ref{reviewcohomology}, we review the Eichler-Shimura cohomology theory for weakly holomorphic modular forms, which were proved by Bruggeman, Choie, and Diamantis \cite{BCD}. And then, by using the result, we prove that there is no nonzero harmonic weak Maass form of half integral weight such that its holomorphic part is zero. In Section \ref{proofs}, we prove Theorem  \ref{determine}.

\section{Modular traces and their modularity} \label{tracesingular}

In this section, we review basic notions of harmonic weak Maass forms, definitions of modular traces of weakly holomorphic modular functions, and the results \cite{BFI, DIT3} on modularity of modular traces of a weakly holomorphic modular function. We follow the notation in \cite{BFI} (see also  \cite{KS, KM, Shi}).

\subsection{Harmonic weak Maass forms} \label{section2}
A harmonic weak Maass form represents a special case of the real analytic modular forms introduced by Maass \cite{Maa}. For $z\in \HH$, let $x$ (resp. $y$) denote the real part (resp. imaginary part) of $z$.
Let \[\Delta_k := -y^2\left(\frac{\partial^2}{\partial x^2} + \frac{\partial^2}{\partial y^2}\right) + iky\left(\frac{\partial}{\partial x}+ i\frac{\partial}{\partial y}\right).\]
For $\gamma= \sm a&b\\c&d\esm\in \Gamma_0(4N)$ and a function $f$ on $\HH$,
we have a slash operator defined by
\[
(f|_{k}\gamma)(z) := \overline{\chi^{}_{\theta}(\gamma)}^{2k}(cz+d)^{-k}f(\gamma z),
\]
where $\gamma z = \frac{az+b}{cz+d}$.

\begin{dfn}
A harmonic weak Maass form of weight $k$ on $\Gamma_0(4N)$ is a smooth function on $\HH$ satisfying the following properties
\begin{enumerate}
\item for all $\gamma\in\Gamma_0(4N)$, $f|_{k}\gamma = f$,
\item $\Delta_k f = 0$,
\item there exists a constant $\delta>0$ such that
\[(cz+d)^{-k}f(\gamma z) = O(e^{\delta y}),\]
as $y\to\infty$ uniformly in $x$, for every element $\gamma = \sm a&b\\c&d\esm\in \SL_2(\ZZ)$.
\end{enumerate}
\end{dfn}

Let $H_{k}(\Gamma_0(4N))$ denote the space of harmonic weak Maass forms of weight $k$ on $\Gamma_0(4N)$. Let $B_k(x) := e^{-x}\int^{\infty}_{-2x} e^{-t}t^{-k}dt$, which converges for $k<1$ and is holomorphically continued in $k$ (for $x\neq 0$) in the same manner as the gamma function.
If $f$ is a harmonic weak Maass form in $H_{k}(\Gamma_0(4N))$, then $f$ has a unique decomposition $f = f^+ + f^-$, where
\begin{equation} \label{pluspart}
f^+(z) = \sum_{n\gg-\infty} a^+(n)e^{2\pi inz}
\end{equation}
and
\begin{equation} \label{minuspart}
f^-(z) =  \sum_{n\neq0} a^-(n)B_k(2\pi ny)e^{2\pi inx} +a^-(0)c(y).
\end{equation}
Here, $c(y) = \log y$ if $k=1$ and $c(y) = y^{1-k}$ otherwise.
It was proved that there is a constant $C>0$ such that
the Fourier coefficients satisfy
\begin{eqnarray*}
a^+(n) &=& O\left(e^{C\sqrt{|n|}}\right),\ n \to \infty,\\
a^-(n) &=& O\left(e^{C\sqrt{|n|}}\right),\ n\to-\infty
\end{eqnarray*}
(see \cite[Lemma 3.4]{BF}).
In this decomposition, $f^+$ (resp. $f^-$) is called the holomorphic part (resp. nonholomorphic part) of $f$.
Remark that this decomposition depends on the choice of the branch.
In this paper, we follow the choice of the branch in \cite{BF}.

\subsection{Quadratic spaces}
Let $(V,Q)$ be the $3$-dimensional quadratic space over $\QQ$ given by
\[
V := \left\{ X = \sm x_1 & x_2 \\ x_3 & -x_1 \esm \in \mathrm{Mat}_2(\QQ) \right\},
\]
with the quadratic form $Q(X) := -N\det(X)$.
We let $G = \mathrm{Spin}(V)$, viewed as an algebraic group over $\QQ$, and write $\bar{G}$ for its image in $\mathrm{SO}(V)$.
The corresponding bilinear form is $(X,Y) = N\mathrm{tr}(XY)$, and its signature is $(2, 1)$.
The group $\mathrm{SL}_2(\QQ)$ acts on $V$ by
\[
g.X := gXg^{-1}
\]
for $X\in V$ and $g\in \SL_2(\QQ)$, which gives rise to isomorphisms $G\cong \mathrm{SL}_2$ and $\bar{G} \cong \mathrm{PSL}_2$.

We realize the associated hermitian symmetric space as the Grassmannian of negative lines in $V(\RR)$
\[
D := \{ z\subset V(\RR)|\ \dim z = 1\ \text{and}\ Q|_z < 0\}.
\]
We identify $D$ with the complex upper half plane $\HH$ as follows. Let $z_0\in D$ be the line spanned by $\sm 0&1\\ -1&0\esm$.
Its stabilizer in $G(\RR)$ is equal to $K = \mathrm{SO}(2)$.
For $z = x + iy\in\HH$, we choose $g_z\in G(\RR)$ such that $g_z i = z$ and put
\[
X(z) := \frac{1}{\sqrt{N}}g_z. \sm 0&1\\-1&0\esm = \frac{1}{\sqrt{N}y}\sm -x&z\bar{z}\\ -1&x\esm \in V(\RR).
\]
We obtain the isomorphism $\Phi:\HH \to D,\ z\mapsto g_z z_0=\RR X(z)$.

Let $L\subset V(\QQ)$ be an even lattice of full rank and write $L'$ for the dual lattice of $L$.
Let $\Gamma$ be a congruence subgroup of $\mathrm{Spin}(L)$ which takes $L$ to itself and acts trivially on the discriminant group $L'/L$.
We set $M := \Gamma\setminus D$.

Since $V$ is isotropic, the modular curve $M$ is a non-compact Riemann surface.
The group $\Gamma$ acts on the set $\mathrm{Iso}(V)$ of isotropic lines in $V$.
The cusps of $M$ correspond to the $\Gamma$-equivalence classes of $\mathrm{Iso}(V)$, with $i\infty$ corresponding to the isotropic line $l_0$ spanned by $u_0 := \sm 0&1\\0&0\esm$.
For $l\in \mathrm{Iso}(V)$, we pick $\sigma_l\in \SL_2(\ZZ)$ such that $\sigma_l l_0 = l$ and set $u_l:= \sigma_l^{-1} u_0$.
We let $\Gamma_l$ be the stabilizer of $l$ in $\Gamma$.
Then
\[
\sigma_l^{-1}\overline{\Gamma}_l \sigma_l = \left\{ \sm 1&k\alpha_l\\ 0&1\esm \big|\ k\in\ZZ\right\}
\]
for some $\alpha_l\in\ZZ_{>0}$.
There is $\beta_l\in\QQ_{>0}$ such that $\beta_l u_l$ is a primitive element of $l\cap L$.
We write $\epsilon_l := \alpha_l / \beta_l$.

\subsection{Modular traces}
For $X\in V$ of negative length $Q(X) = m < 0$, we put $z_X := \RR X\in D$.
We also write $z_X$ for the corresponding point in $\HH$ and we have
\[
z_X = \frac{-b}{2a} + \frac{i\sqrt{|d|}}{2|a|} \in \HH
\]
for $X = \sm b&2c\\-2a&-b\esm$ with $Q(X) = Nd<0$.
For $m\in\QQ^{*}$ and $h\in L'/L$, the group $\Gamma$ acts on $L_{m,h} := \{X\in L+h|\ Q(X) = m\}$ with finitely many orbits.
For a function $f$ on $M$, we define a modular trace of index $m<0$ by
\[
\mathrm{tr}_{m,h}(f) := \sum_{X\in \Gamma\setminus L_{m,h}} \frac{1}{|\bar{\Gamma}_X|} f(z_X).
\]

For $X\in V$ of positive length $Q(X) = m > 0$,
we define a modular trace by using the period $\int_{c(X)}f(z)dz_X$ as follows.
We define a geodesic $c_X$ in $D$ via
\[
c_X := \{z\in D|\ z\perp X\}.
\]
We also write $c_X$ for the corresponding geodesic in $\HH$ and, for $X = \sm b&2c\\-2a&-b\esm$, we have
\[
c_X = \{z \in \HH\ |\ a|z|^2 + b\mathrm{Re}(z) + c = 0\},
\]
where $\mathrm{Re}(z)$ denotes the real part of $z\in\CC$.
We orient the geodesics as follows. For $X = \pm \sm 1&0\\0&-1\esm$, the geodesic $c_X = \pm (0, i\infty)$ is the imaginary axis with the indicated orientation.
The orientation preserving action of $\SL_2(\RR)$ induces an orientation for all $c_X$.
We define the line measure $dz_X$ for $c_X$ by $dz_X := \pm \frac{dz}{\sqrt{m}z}$ for $X = \pm \sqrt{m/N} \sm 1&0\\0&-1\esm$ and then by $dz_{g^{-1}X} = d(gz)_X$ for $g\in \SL_2(\RR)$.
So, for $X = \frac{1}{\sqrt{N}}\sm b&2c\\-2a&-b\esm$, we have
\begin{equation*} \label{differential}
dz_X = \frac{dz}{az^2 + bz+c}.
\end{equation*}
We set $c(X) := \Gamma_X \setminus c_X$, and we use the same symbol for the image of $c(X)$ in $M$.
The stabilizer $\bar{\Gamma}_X$ is either trivial or infinite cyclic.
If $\bar{\Gamma}_X$ is infinite, then $c(X)$ is a closed geodesic in $M$.
Hence, for a continuous function $f$ on $M$, the period $\int_{c(X)} f(z)dz_X$ converges.

If $\bar{\Gamma}_X$ is trivial, then $c(X)$ is an infinite geodesic and we need to regularize the period $\int_{c(X)} f(z)dz_X$.
Assume that $X$ with $Q(X) = m>0$ gives rise to an infinite geodesic in $M$.
Let $f\in M_0^!(\Gamma)$.
For any isotropic line $l$, the function $f_l(z) := f(\sigma_l z)$ can be written as
\[
f_l(z) = \sum_{n\in \frac{1}{\alpha_l}\ZZ} a_l(n)e^{2\pi inz}
\]
with $a_l(n) = 0$ for $n\ll 0$.
Note that $X^{\perp}$ is split over $\QQ$, a rational hyperbolic plane spanned by two rational isotropic lines $l_X$ and $\tilde{l}_X$.
Note that $\tilde{l}_X = l_{-X}$.
We can distinguish these isotropic lines by requiring that $l_X$ represents the endpoint of the oriented geodesic.
The geodesic $c_X$ is explicitly given in $D\cong \HH$ by
\[
c_X = \sigma_{l_X} \{ z\in D|\ z = \Phi(r+it)\ \text{for some $t>0$} \}.
\]
We call $r = \mathrm{Re}(c_X)$ the {\it real part} of the geodesic $c_X$.
Pick a number $c>0$.
We set
\begin{eqnarray*}
\sqrt{m}\int^{\mathrm{reg}}_{c(X)} f(z)dz_X &:=& -a_{l_X}(0)log(c) + \sum_{n\neq0} a_{l_X}(n)e^{2\pi inr}\mathcal{EI}(2\pi nc)\\
&& -a_{\tilde{l}_X}(0)log(\tilde{c}) + \sum_{n\neq0} a_{\tilde{l}_X}(n)e^{2\pi in\tilde{r}}\mathcal{EI}(2\pi n\tilde{c}),
\end{eqnarray*}
where $\tilde{c} = \mathrm{Im}\left(\sigma_{\tilde{l}_X}^{-1}(r+ic)\right)$, $\tilde{r}$ is the real part of $c_{-X}$ and $\mathcal{EI}(w)$ is defined by
\[
\mathcal{EI}(w)  := \int^{\infty}_{w} e^{-t}\frac{dt}{t}.
\]

Then, we define the trace for positive index $m$ and $h\in L'/L$ of $f$ by
\[
\mathrm{tr}_{m,h}(f) := \frac{1}{2\pi} \sum_{X\in \Gamma\setminus L_{m,h}}\int^{\mathrm{reg}}_{c(X)} f(z)dz_X.
\]
Here, if $X$ is a closed geodesic, then $\int^{\mathrm{reg}}_{c(X)} f(z)dz_X$ denotes the usual integral $\int_{c(X)} f(z)dz_X$.
We also define the complementary trace of $f$ for $m\in N(\QQ^*)^2$ and $h\in L'/L$ by
\[
\mathrm{tr}^c_{m,h}(f) := \sum_{X\in \Gamma\setminus L_{m,h}}\left[ \sum_{n<0} a_{l_X}(n)e^{2\pi i\mathrm{Re}(c(X))n} + \sum_{n<0} a_{\tilde{l}_X}(n) e^{2\pi i\mathrm{Re}(c(-X))n}\right].
\]
By Proposition 4.7 of \cite{BF2}, we have
\begin{eqnarray} \label{trc}
\mathrm{tr}^c_{m,h}(f) &=& 2\sqrt{\frac mN}\sum_{l\in\Gamma\setminus \mathrm{Iso}(V)} \epsilon_l \\
\nonumber &&\times \biggl[ \delta_l(m,h) \sum_{n\in \frac 2{\beta_l}\sqrt{\frac mN}\ZZ\atop n<0} a_l(n)e^{2\pi ir_+n} + \delta_l(m,-h) \sum_{n\in \frac 2{\beta_l}\sqrt{\frac mN}\ZZ\atop n<0} a_l(n)e^{2\pi ir_-n}\biggr],
\end{eqnarray}
where $\delta_l(m,h) = 1$ if there exists a vector $X\in L_{m,h}$ such that $c_X$ ends at the cusp $l$ and $\delta_l(m,h)  = 0$ otherwise.
If such $X$ exists, then $r_{\pm}$ is the real part of $\pm X$.
In particular, $\mathrm{tr}^c_{m,h}(f) = 0$ for $m\gg0$.

To define the trace for the index $(0,h)$, we need to truncate $M$ as follows.
For every cusp $l\in \Gamma\setminus \mathrm{Iso}(V)$, we choose sufficiently small neighborhood $U_l$.
We write $q_l := e^{2\pi i(\sigma_l^{-1}z/\alpha_l)}$ with $z\in U_l$ for the local variable around $l\in \bar{M}$, where $\bar{M}$ is the compact Riemann surface obtained by adding a point for each cusp $l\in \Gamma\setminus\mathrm{Iso}(V)$ to $M$.
For $T>0$, we let $U_{1/T} := \left\{w\in \CC\ |\ |w|<\frac{1}{2\pi T}\right\}$. We truncate $M$ by setting
\[
M_T := \bar{M}\setminus \coprod_{[l]\in \mathrm{Iso}(V)} q_l^{-1}(U_{1/T}).
\]
We define the regularized average value of a function $f$ on $M$ by
\[
\int^{\reg}_{M} f(z)\frac{dxdy}{y^2} := \lim_{T\to\infty} \int_{M_T} f(z)\frac{dxdy}{y^2}.
\]
For $h\in L'/L$, we define the trace of $f$ for the index $(0,h)$ by
\[
\mathrm{tr}_{0,h}(f) := -\delta_{h,0}\frac{1}{2\pi} \int^{\mathrm{reg}}_{M} f(z)\frac{dxdy}{y^2}.
\]

\subsection{Modularity of modular traces} \label{workBFI}
In this subsection, we review results on modularity of modular traces of a weakly holomorphic modular function. 
Let $\mathrm{Mp}_2(\RR)$ be the two-fold metapletic cover of $\SL_2(\RR)$ realized as the group of pairs $(g, \phi(g,z))$, where $g = \sm a&b\\c&d\esm \in \SL_2(\RR)$ and $\phi(g,z)$ is a holomorphic square root of the automorphy factor $j(g,z) = cz+d$.
Let $\Gamma'\subset \mathrm{Mp}_2(\RR)$ be the inverse image of $\SL_2(\ZZ)$ under the covering map.
We denote the standard basis of the group algebra $\CC[L'/L]$ by $\{ \mbf{e}_h|\ h\in L'/L\}$.
Recall that there is a Weil representation $\rho_L$ of $\Gamma'$ on the group algebra $\CC[L'/L]$ (see \cite[Section 4]{Bor0} or \cite[Chapter 1.1]{Bru} for explicit formulas).

Let $\Gamma'' \subset \Gamma'$ be a subgroup of finite index.
For $k\in\frac12\ZZ$, we let $A_{k,L}(\Gamma'')$ be the space of $C^{\infty}$-functions $f:\HH\to \CC[L'/L]$ that satisfy
\[
f(\gamma'z) = \phi^{2k}(z)\rho_L(\gamma',\phi) f(z)
\]
for $(\gamma',\phi)\in \Gamma''$.
We denote by $f_h$ the components of $f$ for $h\in L'/L$.
We call a function $f\in A_{k,L}(\Gamma'')$ a harmonic weak Maass form of weight $k$ for $\Gamma''$ with representation $\rho_L$ if $\Delta_k(f) = 0$ and $f$ has at most linear exponential growth at the cusps of $\Gamma''$. In \cite{DIT3}, Duke, Imamo$\bar{\mathrm{g}}$lu, and T\'oth proved that the generating series of $\Tr_d(j_1)$ for $d>0$, defined by
\begin{equation*}
F(z) :=\sum_{d>0} \Tr_d(j_1)q^d
\end{equation*}
is a mock modular form of weight $\frac12$ and multiplier system $\chi^{}_\theta$ on $\Gamma_0(4)$ such that its shadow is the weakly holomorphic modular form $-2g_1$, where $g_1(z) := -q^{-1} + 2 + \sum_{d<0} \Tr_d(j_1)q^{|d|} $. The following theorem is the extension of \cite{DIT3}, by Bruinier, Funke, and  Imamo$\bar{\mathrm{g}}$lu \cite{BFI}, to weakly holomorphic modular functions on modular curves of arbitrary genus.

\begin{thm} \label{mainofBFI} \cite[Theorem 4.1]{BFI}
Let $f\in M_0^!(\Gamma)$ be a weakly holomorphic modular function whose constant coefficients $a_l(0)$ vanishes at all cusps $l$.
Then, the generating series $\hat{G}(f) = \sum_{h\in L'/L}\hat{G}(f)_h \mbf{e}_h$ is a harmonic weak Maass form of weight $\frac12$ for $\Gamma'$ with representation $\rho_L$, where
\begin{eqnarray*}
\hat{G}(f)_h(z) &:=& -2\sqrt{y} \mathrm{tr}_{0,h}(f) + \sum_{m<0}\mathrm{tr}_{m,h}(f)\frac{\mathrm{erfc}(2\sqrt{\pi|m|y})}{2\sqrt{|m|}}e^{2\pi imz}\\
&& + \sum_{m>0} \mathrm{tr}_{m,h}(f)e^{2\pi imz} + 2\sum_{m>0}\mathrm{tr}^c_{Nm^2,h}(f)\left(\int_0^{\sqrt{y}} e^{4\pi Nm^2w^2}dw\right)  e^{2\pi iNm^2z}.
\end{eqnarray*}
Here, $\mathrm{erfc}(w) = \frac{2}{\sqrt{\pi}}\int^{\infty}_{w} e^{-t^2}dt$ is the complementary error function.
\end{thm}

We consider the lattice
\begin{equation} \label{lattice}
L = \left\{ \sm b & c/N \\ a & -b\esm\bigg|\ a,b,c\in\ZZ\right\}.
\end{equation}
Its dual lattice is
\[
L' = \left\{ \sm b & c/N\\ a & -b\esm \bigg|\ a,c\in\ZZ, b\in\frac1{2N}\ZZ\right\}.
\]
Then, given $g = \sum_{h\in L'/L} g_h \mbf{e}_h \in A_{k,L}(\Gamma')$,
the sum
\[
 \sum_{h\in L'/L} g_h(4Nz)
\]
gives a scalar-valued form of weight $k$ for $\Gamma_0(4N)$.
From this, we have the following corollary.

\begin{cor} \label{scalar}
Let $L$ be a lattice given in (\ref{lattice}).
Let $f\in M^!_0(\Gamma_0(N))$ be a weakly holomorphic modular function whose constant coefficients $a_l(0)$ vanish at all cusps $l$.
Then,
\begin{eqnarray*}
\hat{H}(f)(z) &:=& \sum_{d>0} \mathrm{Tr}_d(f)e^{2\pi idz}
-2\mathrm{Tr}_0(f)\sqrt{y} + \sum_{d<0} \mathrm{Tr}_d(f)\frac{\mathrm{erfc}(2\sqrt{\pi|d|y})}{\sqrt{|d|}} e^{2\pi idz}\\
&& + \frac1{\sqrt{N}}\sum_{d>0} \mathrm{Tr}^c_{d^2}(f)\left( \int^{2\sqrt{Ny}}_{0} e^{\pi d^2w^2/N}dw \right) e^{2\pi id^2z}
\end{eqnarray*}
is a harmonic weak Maass form in $H_{\frac12}(\Gamma_0(4N))$, where $\mathrm{Tr}^c_{d^2}(f) = \sum_{h\in L'/L} tr^c_{\frac{d^2}{4N},h}(f)$.
In particular, we have
\[
W(f) := -\frac12 \xi_{\frac12}\left(\hat{H}(f)\right) = -\frac12 \sum_{d>0} \overline{\mathrm{Tr}^c_{d^2}(f)}q^{-d^2}   + \frac12\overline{\mathrm{Tr}_0(f)} +
\sum_{d<0} \overline{\mathrm{Tr}_d(f)}q^{-d}.
\]
\end{cor}

Although this corollary was explained in \cite{BFI}, we will sketch the proof
for the convenience of the reader (see also \cite[Section 6]{BF2}).

\begin{proof} [\bf Proof of Corollary \ref{scalar}]
From the vector-valued harmonic weak Maass form $\hat{G}(f)$ on $\Gamma'$ in Theorem \ref{mainofBFI}, we obtain a scalar-valued harmonic weak Maass form in $H_{\frac12}(\Gamma_0(4N))$
\begin{eqnarray*}
\frac1{2\sqrt{N}} \sum_{h\in L'/L} \hat{G}(f)_h(4Nz) &=&
-2\sqrt{y}\mathrm{tr}_{0,0}(f) + \frac1{4\sqrt{N}}\sum_{m<0} \sum_{h\in L'/L} \mathrm{tr}_{m,h}(f) \frac{\mathrm{erfc}(4\sqrt{\pi|m|Ny})}{\sqrt{|m|}}e^{2\pi i(4Nm)z}\\
&&+ \frac{1}{2\sqrt{N}}\sum_{m>0}\sum_{h\in L'/L} \mathrm{tr}_{m,h}(f)e^{2\pi i(4Nm)z}\\
&& + \frac1{\sqrt{N}}\sum_{m>0} \sum_{h\in L'/L} \mathrm{tr}^c_{Nm^2,h}(f)\left( \int^{2\sqrt{Ny}}_{0} e^{4\pi Nm^2w^2}dw\right) e^{2\pi i(2Nm)^2z}.
\end{eqnarray*}

To $X = \sm b&c/N \\ a & -b\esm \in L'$ with $a,c\in \ZZ$ and $b\in \frac1{2N}\ZZ$, we can correspond an integral binary quadratic form $g_X(x,y) = -aNx^2 + 2Nbxy + cy^2$.
If $Q(X) = m\in \frac1{4N}\ZZ$, then the discriminant of $g_X$ is $4Nm\in \ZZ$.
With this correspondence, if $m<0$, we have
\begin{eqnarray} \label{rel1}
\sum_{h\in L'/L} \mathrm{tr}_{m,h}(f) = 2\mathrm{Tr}_{4Nm}(f).
\end{eqnarray}
The number $2$ comes from the fact that the left hand side counts positive and negative definite binary quadratic forms of discriminant $m$.
If $m>0$, then we have
\begin{equation} \label{rel2}
\sum_{h\in L'/L} \mathrm{tr}_{m,h}(f) = 2\sqrt{N}\mathrm{Tr}_{4Nm}(f).
\end{equation}
The number $2\sqrt{N}$ comes from the fact that $dz_X = 2\sqrt{N}\frac{dz}{g_X(z,1)}$.
Finally, note that
\begin{equation} \label{rel3}
\mathrm{tr}_{0,0}(f) = \mathrm{Tr}_0(f).
\end{equation}
From the relations (\ref{rel1}), (\ref{rel2}), and (\ref{rel3}), we can see that $\frac1{2\sqrt{N}} \sum_{h\in L'/L} \hat{G}(f)_h(4Nz)= \hat{H}(f)(z)$.
\end{proof}

\begin{rmk} \label{squarefreeformula}
If $N$ is a square-free positive integer, then $\Gamma_0^*(N)$ has only one cusp.
Hence, if $f$ is a weakly holomorphic modular function on $\Gamma_0^*(N)$, then, by (\ref{trc}), we have
\[
\mathrm{Tr}^c_{d^2}(f) = 2d\sum_{n<0} a(dn).
\]
\end{rmk}

\section{Regularized Eichler integrals} \label{regular}


Let $g$ be a weakly holomorphic modular form in $M_{k}^{!}(\Gamma_0(4N))$.
Let $z_1\in \HH$ and $z_2, \alpha,\beta,s\in \CC$.
Assume that, for each $\gamma = \sm a&b\\c&d\esm\in\SL_2(\ZZ)$, the function $(cz+d)^{-k}g(\gamma z)$ has a Fourier expansion of the form as in (\ref{fourierq}).
For a complex number $z$, let $\mathrm{Re}(z)$ and $\mathrm{Im}(z)$ denote the real and imaginary part of $z$, respectively.

We consider an integral of the form
\begin{equation} \label{inside}
I_s(g,z_1, z_2, \alpha,\beta) := \int^{i\infty}_{z_1} g(\tau)(\tau-\overline{z_2})^{s+\alpha}e^{2\pi i\beta\tau}d\tau.
\end{equation}
Let us fix  $z_1$ and $z_2$ such that $\mathrm{Im}(z_1) + \mathrm{Im}(z_2)>0$.
For $\mathrm{Re}(\alpha) \ll0$ and $\mathrm{Re}(\beta)\gg0$, the above integral converges absolutely and defines a holomorphic function of $(\alpha,\beta)$.
If, for a fixed $\alpha$ such that $\mathrm{Re}(\alpha)\ll 0$, the integral $I_s(g,z_1, z_2, \alpha,\beta)$ has  analytic continuation at $\beta=0$, then the value of this analytic continuation at $\beta=0$ is denoted by $\left[ I_{s}(g,z_1, z_2, \alpha,\beta)\right]_{\beta=0}$.
If $\left[ I_{s}(g,z_1, z_2, \alpha,\beta)\right]_{\beta=0}$ has  analytic continuation at $\alpha=0$, then the value of this analytic continuation at $\beta=0$ is denoted by $\left[ \left[ I_{s}(g,z_1, z_2, \alpha,\beta)\right]_{\beta=0} \right]_{\alpha=0}$.
If these two continuations exist, we define the
 {\it  regularized Eichler integral associated with $g$} by
\begin{equation} \label{regularEichler}
\int^{\reg}_{[z_1,i\infty]} g(\tau)(\tau-\overline{z_2})^{s}d\tau := \left[ \left[ I_{s}(g,z_1, z_2, \alpha,\beta)\right]_{\beta=0} \right]_{\alpha=0}.
\end{equation}

More generally, if we let $r\in\QQ$ and $\gamma = \sm a&b\\c&d\esm\in\SL_2(\ZZ)$ such that $r = \gamma(i\infty)$,
then the {\it regularized Eichler integral associated with $g$ for a cusp $r$} is defined by
\begin{equation} \label{othercusp}
\int^{\reg}_{[z_1,r]} g(\tau)(\tau-\overline{z_2})^{s}d\tau := \left[ \left[ \int^{i\infty}_{\gamma^{-1}(z_1)} (c\tau+d)^{-2+\alpha}g(\gamma\tau)(\gamma\tau-\overline{z_2})^{s+\alpha}e^{2\pi i\beta\tau}d\tau \right]_{\beta=0} \right]_{\alpha=0}.
\end{equation}
Since $\chi_{\theta}(T) = 1$, this definition is independent of the choice of $\gamma$, i.e.,  the right hand side of (\ref{othercusp}) remains the same
when we replace $\gamma$ by $\pm \gamma T^n$ for $n\in\ZZ$, where $T = \sm 1&1\\0&1\esm$.

For two distinct cusps $r_1$ and  $r_2$, we define
\begin{equation} \label{twocusp}
\int^{\reg}_{[r_1,r_2]} g(\tau)(\tau-\overline{z_2})^{s}d\tau
:= -\int^{\reg}_{[z_1,r_1]} g(\tau)(\tau-\overline{z_2})^{s}d\tau + \int^{\reg}_{[z_1,r_2]} g(\tau)(\tau-\overline{z_2})^{s}d\tau,
\end{equation}
for any $z_1\in\HH$. Let us note that a regularized integral in (\ref{twocusp}) is independent of the choice of $z_1$ (see Lemma \ref{independent} in Section \ref{Well-definedness of regularized Eichler integrals}). In this section, we prove the following theorem.

\begin{thm} \label{maintheoremsection3}
Let $z_1\in\HH$, $r\in\QQ\cup\{i\infty\}$ and $z_2\in \HH$.
Let $f\in H_{2-k}(\Gamma_0(4N))$ and $g = \xi_{2-k}(f) \in M^!_{k}(\Gamma_0(4N))$. If $k\neq1$, then the regularized Eichler integral
$$
\int^{\reg}_{[z_1,r]} g(\tau)(\tau-\overline{z_2})^{k-2}d\tau
$$
is well defined for $z_1$ and $z_2$.
Moreover, we have
\[
\xi_{2-k}\left(\overline{\int^{\reg}_{[z,i\infty]} g(\tau)(\tau-\overline{z})^{k-2}d\tau}\right) = -(2i)^{k-1}g(z).
\]
and
\[
(f^- - f^-|_{2-k}\gamma)(z) = -\frac{1}{(-2i)^{k-1}} \overline{\int^{\reg}_{[\gamma^{-1}
(i\infty),i\infty]} g(\tau)(\tau-\bar{z})^{k-2}d\tau}
\]
for each $\gamma\in \Gamma_0(4N)$ and $z_1\in \HH$.
\end{thm}

\begin{rmk}
It was noted by Bruggeman that one can also use the method in Section 3 of \cite{BCD} to define the regularized Eichler integrals
and to prove their properties obtained in this paper.
\end{rmk}

\subsection{Asymptotic behavior of a generalized incomplete gamma function $A(x,y,s)$}
To prove the well-definedness of the regularized Eichler integral $\int^{\reg}_{[z_1, i\infty]} g(\tau)(\tau-\overline{z_2})^s d\tau$, we study the function
\[
A(x,y,s) := \int^{\infty}_{x} e^{-t}(t+iy)^s dt
\]
for $x, y\in \RR$ and $s\in \CC$.
Note that if $y\neq 0$, then the function $A(x,y,s)$ is well defined for $x\in\RR$ and $y\in \RR$.
If $y = 0$, then we define $A(x,y, s)$ as the analytic continuation of
the incomplete gamma function $\Gamma(s+1, x)$.
This function appears in series expressions of the regularized Eichler integrals (see Lemma \ref{generalpoint}).

It is known that the incomplete gamma function satisfies the following asymptotic properties
\begin{equation} \label{incompletegamma}
\Gamma(s+1, x) \sim e^{-x} x^{s}
\end{equation}
as $x\to \pm\infty$ (for example see \cite[Section 3]{BF}).
In the following lemma, as $x\to \pm\infty$, we prove the asymptotic properties of $A(x,y,s)$ for  general $y$ and $s$.

\begin{lem} \label{gammaasymptotic}
We assume the above notation.
\begin{enumerate}
\item For each fixed $y\in \RR$, we have $A(x, y, s) \sim e^{-x}(x+iy)^s$ as $x \to \pm\infty$.
\item Let $c(x) = ax + b$ for $a, b\in \RR$. Then we have $A(x, c(x), s) \sim e^{-x}(x+ic(x))^s$ as $x\to\infty$.
\end{enumerate}
\end{lem}

\begin{proof} [\bf Proof of Lemma \ref{gammaasymptotic}]
First, we consider the case when $x$ goes to $\infty$.
In this case, we only need to prove the part (2) since the part (1) comes by setting $c(x) = y$.
If we use the integration by parts, $A(x,c(x), s)$ is equal to
\begin{eqnarray} \label{axys0}
\nonumber && e^{-x}(x+ic(x))^{s} + s e^{-x}(x+ic(x))^{s-1} + \cdots +
s(s-1)\cdots (s-(n-2))e^{-x}(x+ic(x))^{s-(n-1)}\\
 &&+ s(s-1)\cdots (s-(n-1))\int^{\infty}_{x} e^{-t}(t+ic(x))^{s-n}dt
\end{eqnarray}
for $x>0$ and $n>0$.
We take $n$ such that $n\geq 2$ and $\mathrm{Re}(s)-n+1<0$.
Note that
\begin{eqnarray*}
e^{x} \left|\int^{\infty}_{x} e^{-t}(t+ic(x))^{s-n}dt\right| &\leq& \int^{\infty}_{x}\left| e^{x-t}(t+ic(x))^{s-n}\right| dt \leq \int^{\infty}_{x}\left| (t+ic(x))^{s-n}\right| dt
\end{eqnarray*}
and
\begin{eqnarray} \label{complexexponent}
|(t+ic(x))^{s-n}| = |t+ic(x)|^{\mathrm{Re}(s)-n} e^{-\mathrm{arg}(t+ic(x))\mathrm{Im}(s)}.
\end{eqnarray}
For $t\in [x,\infty)$, we have
\[
0< t \leq |t+ ic(x)|\ \text{and}\ |\mathrm{arg}(t+ic(x))| \leq |\mathrm{arg}(x+ic(x))|
\]
and
\begin{equation} \label{axys}
\int^{\infty}_{x}\left| (t+ic(x))^{s-n}\right| dt \leq  e^{|\mathrm{arg}(x+ic(x))\mathrm{Im}(s)|}  \int^{\infty}_x t^{\mathrm{Re}(s)-n}dt.
\end{equation}
The right hand side of (\ref{axys}) is equal to
\begin{equation} \label{axys2}
e^{|\mathrm{arg}(x+ic(x))\mathrm{Im}(s)|} \left[ \frac{1}{\mathrm{Re}(s)-n+1}t^{\mathrm{Re}(s)-n+1}\right]^{\infty}_{x} = -\frac{ e^{|\mathrm{arg}(x+ic(x))\mathrm{Im}(s)|}} {\mathrm{Re}(s)-n+1} x^{\mathrm{Re}(s)-n+1}.
\end{equation}
On the other hand, we have
\begin{eqnarray*}
\lim_{x\to\infty} \frac{e^{-x}(x+ic(x))^{s-j}}{e^{-x}(x+ic(x))^{s}} &=& \lim_{x\to\infty} \frac{1} {(x+ic(x))^j} = 0,\ 1\leq j\leq n-1,
\end{eqnarray*}
and
\begin{eqnarray*}
\lim_{x\to\infty} \frac{\left|\int^{\infty}_{x} e^{-t}(t+ic(x))^{s-n}dt\right|}{|e^{-x}x^s|} = \lim_{x\to\infty} \frac{e^x\left|\int^{\infty}_{x} e^{-t}(t+ic(x))^{s-n}dt\right|} {x^{\mathrm{Re}(s)}} \leq \frac{ e^{|\pi\mathrm{Im}(s)|}} {\mathrm{Re}(s)-n+1} \lim_{x\to\infty} \frac{-1}{x^{n-1}} = 0.
\end{eqnarray*}
Here, we used (\ref{axys2}) and the fact that $n\geq 2$.
If we combine all these limits with (\ref{axys0}), then we have
\begin{eqnarray*}
\lim_{x\to\infty} \frac{A(x,y,s)}{e^{-x}(x+ic(x))^s} = 1.
\end{eqnarray*}

Now, we will prove the part (1) as $x\to-\infty$.
We may assume that $y\neq 0$ since the asymptotic properties of $A(x,0,s)$ comes directly from (\ref{incompletegamma}).
Note that the integral
\[
\int^{\infty}_{-1} e^{-t}(t+iy)^sdt
\]
is well defined if $y\neq0$.
Therefore, it suffices to show that
\begin{equation} \label{axysminus}
\lim_{x<-1 \atop x\to-\infty} \frac{\int^{-1}_x e^{-t}(t+iy)^sdt}{e^{-x}x^s} = 1.
\end{equation}
We will use the similar argument as above.
Suppose that $x<-1$.
Note that the integral
\[
\int^{-1}_x e^{-t}(t+iy)^sdt
\]
is equal to
\begin{eqnarray} \label{axysminus2}
&& e^{-x}(x+iy)^{s} + s e^{-x}(x+iy)^{s-1} + \cdots +
s(s-1)\cdots (s-(n-2))e^{-x}(x+iy)^{s-(n-1)}\\
\nonumber &&+ s(s-1)\cdots (s-(n-1))\int^{-1}_{x} e^{-t}(t+iy)^{s-n}dt\\
\nonumber &&-e(-1+iy)^s -se(-1+iy)^s - \cdots -s(s-1)\cdots (s-(n-1))e(-1+iy)^{s-(n-1)}.
\end{eqnarray}
We take $n$ such that $n\geq2$ and $\mathrm{Re}(s) - n + 1<0$.
The terms in the last line of (\ref{axysminus2}) are constant terms with respect to the $x$ variable. Therefore, they do not contribute in the computation of the limit (\ref{axysminus}) since $\lim_{x\to-\infty} |e^{-x}x^s| = \infty$.
One can also see that
\[
\lim_{x\to-\infty} \frac{e^{-x}(x+iy)^{s-j}}{e^{-x}(x+iy)^s} = 0,\ 1\leq j\leq n-1.
\]
Thus, it is enough to show that
\begin{equation} \label{axysminus3}
\lim_{x\to-\infty} \frac{\left|\int^{-1}_{x} e^{-t}(t+iy)^{s-n}dt\right|}{|e^{-x}(x+iy)^s|} = 0.
\end{equation}
For $t\in [x,-1]$, we have
\[
0< -t \leq |t+iy|\ \text{and}\ |\mathrm{arg}(t+iy)| \leq |\mathrm{arg}(x+iy)|.
\]
If we use (\ref{complexexponent}), then we obtain
\begin{equation} \label{axysminus4}
e^x \left|\int^{-1}_{x} e^{-t}(t+iy)^{s-n}dt\right| \leq \int^{-1}_x |(t+iy)^{s-n}|dt \leq e^{|\mathrm{arg}(x+iy)\mathrm{Im}(s)|}\int^{-1}_x (-t)^{\mathrm{Re}(s)-n}dt.
\end{equation}
The right hand side of (\ref{axysminus4}) is equal to
\[
\frac{e^{|\mathrm{arg}(x+iy)\mathrm{Im}(s)|}}{\mathrm{Re}(s)-n+1}((-x)^{\mathrm{Re}(s)-n+1}-1).
\]
Therefore, the equation (\ref{axysminus3}) holds since the fact that $n\geq2$ implies that
\[
\lim_{x\to-\infty} \frac{1}{x^{n-1}} = 0.
\]
This completes the proof.
\end{proof}

In the following lemma, we prove the uniform convergence of a certain series by using Lemma \ref{gammaasymptotic}.

\begin{lem} \label{uniform}
Let $E\subset \RR$ be bounded.
Suppose that $a(n)$ is a sequence with $a(n) = O(e^{C\sqrt{n}})$ for some $C>0$ as $n\to\infty$.
Then the series
\begin{equation} \label{Aseries}
\sum_{n=1}^{\infty} a(n) e^{nb(t)} A(n(b(t)+\alpha), nc(t), s)
\end{equation}
converges uniformly on $E$, where $b(t)$ and $c(t)$ are continuous real-valued functions on $E$, and $\alpha$ is a positive real number.
\end{lem}

\begin{proof} [\bf Proof of Lemma \ref{uniform}]
Let $\beta$ be a positive real number.
Since $a(n) = O(e^{C\sqrt{n}})$ as $n\to\infty$, the series
\[
\sum_{n=1}^\infty |a(n)|e^{-n(\alpha+\beta)}
\]
converges.
Therefore, for given $\epsilon>0$, there exists $N_1\in\NN$ such that
\[
\sum_{m=n}^\infty |a(m)|e^{-m(\alpha+\beta)} < 2\epsilon
\]
for all $n\geq N_1$.
Note that by Lemma \ref{gammaasymptotic} we have
\[
\lim_{n\to\infty} \frac{A(n(b(t)+\alpha), nc(t), s)}{e^{n(b(t)+\alpha)}e^{n\beta}} = \lim_{n\to\infty} \frac{(n(b(t)+\alpha)+inc(t))^s}{e^{n\beta}} = 0.
\]
Since $E$ is bounded, the sets $\{ b(t)\ |\ t\in E\}$ and $\{ c(t)\ |\ t\in E\}$ are also bounded.
This means that there exists $N_2\in\NN$ such that
\[
\left| \frac{A(n(b(t)+\alpha), nc(t), s)}{e^{n(b(t)+\alpha)}e^{n\beta}}\right| < \frac12
\]
for all $n\geq N_2$ and all $t\in E$.
Therefore, for all $n\geq \mathrm{max}(N_1, N_2)$, we have
\[
\left| \sum_{m=n}^\infty a(m)e^{nb(t)}A(n(b(t)+\alpha), nc(t), s)\right| < \frac12 \sum_{m=n}^\infty |a(m)|e^{-m(\alpha+\beta)}< \epsilon.
\]
Hence, the series (\ref{Aseries}) converges uniformly.
\end{proof}

\subsection{Well-definedness of regularized Eichler integrals}\label{Well-definedness of regularized Eichler integrals}
The following lemma shows that  the regularized Eichler integral associated with $g$ defined in (\ref{regularEichler})
is well defined for  $z_1\in \HH$ and $z_2\in \CC$ if $s\neq1$ and $\mathrm{Im}(z_1) + \mathrm{Im}(z_2) > 0$.

\begin{lem} \label{generalpoint}
Let $z_1\in\HH$ and $z_2\in \CC$.
Suppose that $s\neq-1, \mathrm{Im}(z_1) + \mathrm{Im}(z_2) >0$ and
$g(z) = \sum_{n\geq n_0} b(n)e^{2\pi inz} \in M^!_{k}(\Gamma_0(4N))$ for $n_0\in\ZZ$.
Then, the regularized Eichler integral associated with $g$ 
 is well defined, and can be written as
\begin{eqnarray*}
\int^{\reg}_{[z_1,i\infty]} g(\tau)(\tau-\overline{z_2})^{s}d\tau &=&  \left(\frac i{2\pi}\right)^{s+1}\sum_{n\gg-\infty\atop n\neq0} \frac{b(n)}{n^{s+1}} e^{2\pi inx_1} e^{2\pi ny_2} A(2\pi n(y_1+y_2), 2\pi n(x_2-x_1), s)\\
&&-b(0)\frac{i^{s+1}(y_1+y_2+i(x_2-x_1))^{s+1}}{s+1},
\end{eqnarray*}
where $x_i$ (resp. $y_i$) denotes the real part (resp. imaginary part) of $z_i$ for $i\in \{1,2\}$.
\end{lem}

\begin{proof} [\bf Proof of Lemma \ref{generalpoint}]
By the Fourier expansion of $g$,  we have
\begin{eqnarray} \label{I1}
 I_s(g,z_1, z_2, \alpha,\beta) = \int^{i\infty}_{z_1} \sum_{n\gg-\infty}b(n)e^{2\pi in\tau}(\tau-\overline{z_2})^{s+\alpha}e^{2\pi i\beta\tau}d\tau.
\end{eqnarray}
Suppose that $n+\re(\beta)>0$ for all $n\geq n_0$ and $\re(\alpha)< -\re(s)-1$.
Then, the integral in the right hand side of (\ref{I1}) is well defined and
we can exchange the order of summation and integration in (\ref{I1}).
Thus, we have
\begin{eqnarray} \label{computeintegral}
 I_s(g,z_1, z_2, \alpha,\beta)  &=& \sum_{n\gg-\infty\atop n\neq0} b(n)\int^{i\infty}_{z_1} e^{2\pi in\tau}(\tau-\overline{z_2})^{s+\alpha}e^{2\pi i\beta\tau}d\tau\\
\nonumber && + b(0)\int^{i\infty}_{z_1}(\tau-\overline{z_2})^{s+\alpha}e^{2\pi i\beta\tau}d\tau.
\end{eqnarray}
We denote the first summand (resp. the second summand) in (\ref{computeintegral}) by
 $S_1(z_1,z_2,\alpha,\beta)$ (resp. $S_2(z_1,z_2,\alpha,\beta)$).

Let us note that $\re(\alpha)<-\re(s)-1$.
We have
\begin{eqnarray*}
S_2(z_1,z_2, \alpha,0) &=& b(0) \int^{i\infty}_{z_1} (\tau-\overline{z_2})^{s+\alpha}d\tau = b(0) \left[ \frac1{s+1+\alpha}(\tau-\overline{z_2})^{s+1+\alpha}\right]_{z_1}^{i\infty}\\
 &=& -b(0) \frac{(z_1-\overline{z_2})^{s+1+\alpha}}{s+1+\alpha} = -b(0) \frac{i^{s+1+\alpha}(y_1+y_2+i(x_2-x_1))^{s+1+\alpha}}{s+1+\alpha}.
\end{eqnarray*}
On, the other hand, by applying the change of variables $\tau = x_1 + i(t-y_2)$, we obtain
\begin{eqnarray*}
S_1(z_1,z_2,\alpha,\beta)  &=& i^{s+1+\alpha}\sum_{n\gg-\infty\atop n\neq0} b(n)e^{2\pi i(n+\beta)(x_1-iy_2)}\int^{\infty}_{y_1+y_2}e^{-2\pi (n+\beta)t}(t+i(x_2-x_1))^{s+\alpha}dt\\
&=&  \left(\frac i{2\pi}\right)^{s+1+\alpha}\sum_{n\gg-\infty\atop n\neq0} \frac{b(n)}{(n+\beta)^{s+\alpha+1}} e^{2\pi i(n+\beta)(x_1-iy_2)}\\
&&\times A(2\pi(n+\beta)(y_1+y_2),2\pi(n+\beta)(x_2-x_1),s+\alpha).
\end{eqnarray*}
Let us note that
\begin{equation} \label{coefficientgrowth}
b(n) = O\left(e^{C\sqrt{|n|}}\right),\ n\to\infty
\end{equation}
for a constant $C>0$.
If we use Lemma \ref{gammaasymptotic} with (\ref{coefficientgrowth}), then,
for a fixed $\alpha$ such that $\re(\alpha)<-\re(s)-1$, the function $I_s(f,z_1,z_2,\alpha,\beta)$ has an analytic continuation to $\beta=0$
as
\begin{align*}
\left[ I_{s}(f,z_1, z_2, \alpha,\beta)\right]_{\beta=0}=&\left(\frac i{2\pi}\right)^{s+1+\alpha}\sum_{n\gg-\infty\atop n\neq0} \frac{b(n)}{n^{s+\alpha+1}} e^{2\pi in(x_1-iy_2)} A(2\pi n(y_1+y_2),2\pi n(x_2-x_1),s+\alpha)\\
&\quad -b(0)\frac{i^{s+1+\alpha}(y_1+y_2+i(x_2-x_1))^{s+1+\alpha}}{s+1+\alpha}.
\end{align*}

Finally, if $s\neq-1$, then we have
\begin{align*}
\left[ \left[ I_{s}(f,z_1, z_2, \alpha,\beta)\right]_{\beta=0} \right]_{\alpha=0}=&\left(\frac i{2\pi}\right)^{s+1}\sum_{n\gg-\infty\atop n\neq0} \frac{b(n)}{n^{s+1}} e^{2\pi in(x_1-iy_2)}A(2\pi n(y_1+y_2),2\pi n(x_2-x_1),s)\\
&\ -b(0)\frac{i^{s+1}(y_1+y_2+i(x_2-x_1))^{s+1}}{s+1}.
\end{align*}
Thus, the regularized Eichler integral
\[
\int^{\reg}_{[z_1,i\infty]} g(\tau)(\tau-\overline{z_2})^{s}d\tau = \left[ \left[ I_{s}(g,z_1, z_2, \alpha,\beta)\right]_{\beta=0} \right]_{\alpha=0}
\]
 is well defined for $z_1\in\HH$, $z_2,s\in\CC$ with $y_1 + y_2 > 0$ and  $s\neq-1$.
\end{proof}

\begin{rmk} \label{nonholo}
\begin{enumerate}
\item If the integral $\int^{i\infty}_{z_1} g(\tau)(\tau-\overline{z_2})^{s}d\tau$ is well defined, then it is the same as the regularized Eichler integral of $g$ in (\ref{regularEichler}).

\item If $b(0) = 0$, then the regularized Eichler integral of $g$ in (\ref{regularEichler}) is the same as the analytic continuation of $I_s(g,z_1,z_2, 0,\beta)$ at $\beta = 0$, which is denoted by
$\left[ I_s(g,z_1,z_2, 0,\beta)\right]_{\beta=0}$.

\item In general, the regularized Eichler integral
$\int^{\reg}_{[z_1,i\infty]} g(\tau)(\tau-\overline{z_2})^s d\tau$
is not the same as the limit of integrals $\lim_{y\to\infty} \int_{z_1}^{z_1+iy} g(\tau)(\tau-\overline{z_2})^sd\tau$.
In fact, this limit does not converge if $g$ has a principal part at the cusp $i\infty$.
However, if $g$ is a cusp form, then the limit converges and is equal to $\int^{\reg}_{[z_1,i\infty]} g(\tau)(\tau-\overline{z_2})^s d\tau$.

\item
Although we have defined $M^!_{k}(\Gamma_0(4N))$ with the multiplier system $\chi_{\theta}^{2k}$, the statement of Lemma \ref{generalpoint} is valid for all multiplier system $\chi$ on $\Gamma_0(4N)$ of weight $k$ with $\chi(T)=1$.

\item
If we set $z_1 = z_2 = z\in\HH$, then we obtain
\begin{eqnarray*}
\overline{\int^{\reg}_{[z,i\infty]} g(\tau)(\tau-\overline{z})^{k-2}d\tau} &=& \left(\frac{-i}{2\pi}\right)^{k-1}\sum_{n\gg-\infty\atop n\neq0} \overline{\left(\frac{b(n)}{n^{k-1}}\right)}e^{2\pi i(-n)x}B_{2-k}\left(-2\pi ny\right)\\
&&-\frac{\overline{b(0)}(-2i)^{k-1}}{k-1}y^{k-1}.
\end{eqnarray*}
By the definition of $B_k(x)$, we have
\[
e^{2\pi i(-n)x}B_{2-k}\left(-2\pi ny\right) = e^{2\pi i(-n)z}\int^{\infty}_{4\pi ny}e^{-t}t^{k-2}dt
\]
for $n\neq0$.
Since $\frac{\partial}{\partial \bar{z}} = \frac12\left(\frac{\partial}{\partial x} + i\frac{\partial}{\partial y}\right)$, we obtain
\[
\xi_{2-k}\left(\overline{\int^{\reg}_{[z,i\infty]} g(\tau)(\tau-\overline{z})^{k-2}d\tau}\right) = -(2i)^{k-1}g(z).
\]
Note that the Fourier expansion of $\overline{\int^{\reg}_{[z,i\infty]} g(\tau)(\tau-\overline{z})^{k-2}d\tau}$  is of the same form with that  of the nonholomorphic part of a harmonic weak Maass form of weight $2-k$ (see (\ref{minuspart})). 
In other words, we have
\begin{equation} \label{nonholomorphicpart}
f^-(z) = \frac{-1}{(-2i)^{k-1}} \overline{\int^{\reg}_{[z,i\infty]} \xi_{2-k}(f)(\tau)(\tau-\overline{z})^{k-2}d\tau}
\end{equation}
for $f\in H_{2-k}(\Gamma_0(4N))$.
\end{enumerate}
\end{rmk}

The following lemma shows that the regularized Eichler integral defined in (\ref{othercusp}) is well defined on certain conditions.

\begin{lem} \label{welldefinedforcusp}
\begin{enumerate}
\item 
The regularized Eichler integral
$
\int^{\reg}_{[z_1,r]} g(\tau)(\tau-\overline{z_2})^{k-2}d\tau
$
is well defined for $z_1\in \HH$ and $z_2\in \CC$ with $z_2\neq r$ and $\mathrm{Im}(\gamma^{-1}z_1) + \mathrm{Im}(\gamma^{-1}z_2) > 0$,
and equal to
\[
(-c\overline{z_2}+a)^{k-2}\int^{\reg}_{[\gamma^{-1}z_1,i\infty]} (c\tau+d)^{-k}g(\gamma\tau)(\tau-\overline{\gamma^{-1}z_2})^{k-2}d\tau,
\]
where $\gamma = \sm a&b\\c&d\esm \in \SL_2(\ZZ)$ satisfies $r = \gamma(i\infty)$.
In particular, if $\gamma\in \Gamma_0(4N)$ and $z_2\in\HH$, then the regularized Eichler integral is equal to
\[
\chi^{2k}_{\theta}(\gamma)(-c\overline{z_2}+a)^{k-2}\int^{\reg}_{[\gamma^{-1}z_1,i\infty]} g(\tau)(\tau-\overline{\gamma^{-1}z_2})^{k-2}d\tau.
\]

\item If $z_2 = r$ and $\re(s)>k-1$, then the regularized Eichler integral (\ref{othercusp})
 is well defined for $z_1\in \HH$ such that $\re(z_1) = r$. Furthermore, it has a meromorphic continuation to all $s\in\CC$, which has a pole at $s = k-1$.
\end{enumerate}
\end{lem}

\begin{proof} [\bf Proof of Lemma \ref{welldefinedforcusp}]
(1)
For any $\tau\in \HH$ and $z\in\HH$, we obtain
\begin{equation} \label{determinerho}
(\gamma\tau - \overline{\gamma z})^{k-2+\alpha} =
 (\tau-\bar{z})^{k-2+\alpha} (c\tau+d)^{-(k-2+\alpha)} \left(\frac{1}{c\bar{z}+d}\right)^{k-2+\alpha}.
\end{equation}

Therefore, we have
\begin{eqnarray*}
&&\int^{\reg}_{[z_1,r]} g(\tau)(\tau-\overline{z_2})^{k-2}d\tau = \left[ \left[ \int^{i\infty}_{\gamma^{-1}(z_1)} (c\tau+d)^{-2+\alpha}g(\gamma\tau)(\gamma\tau-\overline{z_2})^{k-2+\alpha}e^{2\pi i\beta\tau}d\tau \right]_{\beta=0} \right]_{\alpha=0}\\
&&=  \left[ \left[ \int^{i\infty}_{\gamma^{-1}(z_1)} \left(\frac{1}{c\overline{\gamma^{-1}z_2}+d}\right)^{k-2+\alpha}
(c\tau+d)^{-k}g(\gamma\tau)(\tau-\overline{\gamma^{-1}z_2})^{k-2+\alpha}e^{2\pi i\beta\tau}d\tau \right]_{\beta=0} \right]_{\alpha=0}\\
&&= (-c\overline{z_2}+a)^{k-2}\left[ \left[ \int^{i\infty}_{\gamma^{-1}(z_1)}
(c\tau+d)^{-k}g(\gamma\tau)(\tau-\overline{\gamma^{-1}z_2})^{k-2+\alpha}e^{2\pi i\beta\tau}d\tau \right]_{\beta=0} \right]_{\alpha=0}\\
&&= (-c\overline{z_2}+a)^{k-2}\int^{\reg}_{[\gamma^{-1}z_1,i\infty]} (c\tau+d)^{-k}g(\gamma\tau)(\tau-\overline{\gamma^{-1}z_2})^{k-2}d\tau.
\end{eqnarray*}
By Lemma \ref{generalpoint}, this integral is well defined for $z_1\in \HH$ and $z_2\in \CC$
if $z_2 \neq r$ and $\mathrm{Im}(\gamma^{-1}z_1) + \mathrm{Im}(\gamma^{-1}z_2) > 0$.

(2) By the definition of the regularized Eichler integral in (\ref{othercusp}), we have
\[
\int^{\reg}_{[z_1,r]} g(\tau)(\tau-r)^{s}d\tau = \left[ \left[ \int^{i\infty}_{\gamma^{-1}(z_1)} (c\tau+d)^{-2+\alpha}g(\gamma\tau)(\gamma\tau-r)^{s+\alpha}e^{2\pi i\beta\tau}d\tau \right]_{\beta=0} \right]_{\alpha=0},
\]
where $\gamma = \sm a&b\\c&d\esm \in \SL_2(\ZZ)$ satisfying $r = \gamma(i\infty)$.
Let $z_1 = r+it$ for $t>0$.
By applying the change of variables $w = \gamma\tau -r$, we obtain
\begin{eqnarray*}
&&\int^{i\infty}_{\gamma^{-1}(z_1)} (c\tau+d)^{-2+\alpha}g(\gamma\tau)(\gamma\tau-r)^{s+\alpha}e^{2\pi i\beta\tau}d\tau\\
&&= \int^{0}_{z_1-r} (-cw)^{2-\alpha}g(w+r)w^{s+\alpha}e^{2\pi i\beta\gamma^{-1}(w+r)}(-cw)^{-2}dw\\
&&= c^{-\alpha}e^{\pi i\alpha} \int_{it}^0 g(w+r)w^{s}e^{2\pi i\beta\gamma^{-1}(w+r)}dw.
\end{eqnarray*}
Then, by the change of variables $z = -\frac 1w$, we have
\begin{eqnarray*}
&&\int_{it}^0 g(w+r)w^{s}e^{2\pi i\beta\gamma^{-1}(w+r)}dw= \int_{i/t}^{i\infty} g\left(-\frac1z+r\right)\left(-\frac1z\right)^{s}e^{2\pi i\beta\gamma^{-1}(-\frac1z+r)}\frac{dz}{z^2}.
\end{eqnarray*}
Note that
\[
c\left(\frac z{c^2}-\frac dc\right)+d = \frac zc
\]
and
\[
\gamma\left( \frac z{c^2}-\frac dc\right) =
\frac{a\left(\frac z{c^2}-\frac dc\right) + b}{c\left(\frac z{c^2}-\frac dc\right)+d} = -\frac 1z + \frac ac = -\frac 1z + r.
\]
Therefore, we obtain
\begin{eqnarray} \label{computation}
\nonumber && \int_{i/t}^{i\infty} g\left(-\frac1z+r\right)\left(-\frac1z\right)^{s}e^{2\pi i\beta\gamma^{-1}(-\frac1z+r)}\frac{dz}{z^2}\\
&&=  \int_{i/t}^{i\infty} \left(c\left(\frac z{c^2}-\frac dc\right)+d\right)^{-k}
g\left(\gamma\left( \frac z{c^2}-\frac dc\right)\right)
\biggl(\frac zc\biggr)^k \left(-\frac1z\right)^{s}e^{2\pi i\beta(\frac z{c^2}-\frac dc)}\frac{dz}{z^2}.
\end{eqnarray}
If we employ the Fourier expansion of $(cz+d)^{-k}g(\gamma z)$ in (\ref{fourierq}),
then (\ref{computation}) is equal to
\begin{eqnarray*}
\frac{1}{c^k} i^{2s} \sum_{n\gg-\infty} b_\gamma(n)e^{2\pi i\left( \frac{n+\kappa_\gamma}{\lambda_\gamma} + \beta \right) \left(-\frac dc\right)} \int^{i\infty}_{i/t} e^{2\pi i\left( \frac{n+\kappa_\gamma}{\lambda_\gamma} + \beta \right) \left(\frac{z}{c^2}\right)} z^{k-s-2}dz.
\end{eqnarray*}
If $n+\kappa_\gamma\neq 0$, then, by the change of variables $y = z/i$, we obtain
\begin{eqnarray*}
&&\int^{i\infty}_{i/t} e^{2\pi i\left( \frac{n+\kappa_\gamma}{\lambda_\gamma} + \beta \right) \left(\frac{z}{c^2}\right)} z^{k-s-2}dz = i^{k-s-1} \int^{\infty}_{1/t} e^{-2\pi \left( \frac{n+\kappa_\gamma}{\lambda_\gamma} + \beta \right) \left(\frac{y}{c^2}\right)} y^{k-s-2}dy\\
&&= \Gamma\left(k-s-1, \frac{2\pi}{c^2t}\left(\frac{n+\kappa_\gamma}{\lambda_\gamma}+\beta\right)\right)\frac{ i^{k-s-1}}{\left(\frac{2\pi}{c^2} \left(\frac{n+\kappa_\gamma}{\lambda_\gamma}+\beta\right)\right)^{k-s-1}}.
\end{eqnarray*}
By the same argument as in the proof of Lemma \ref{generalpoint}, we obtain
\begin{eqnarray*}
&&\left[ \int^{i\infty}_{\gamma^{-1}(z_1)} (c\tau+d)^{-2+\alpha}g(\gamma\tau)(\gamma\tau-r)^{s+\alpha}e^{2\pi i\beta\tau}d\tau \right]_{\beta=0}\\
&&= c^{-\alpha} \frac{1}{c^k} (c^2)^{k-s-1} e^{\pi i\alpha}
i^{k+s-1} \sum_{n\gg-\infty\atop n+\kappa
_\gamma\neq 0} \frac{b_\gamma(n)e^{2\pi i\left( \frac{n+\kappa_\gamma}{\lambda_\gamma}\right) \left(-\frac dc\right)}}{\left(2\pi \left(\frac{n+\kappa_\gamma}{\lambda_\gamma}\right)\right)^{k-s-1}}
 \Gamma\left(k-s-1, \frac{2\pi}{c^2t}\left(\frac{n+\kappa_\gamma}{\lambda_\gamma}\right)\right)
 \\
 && +  c^{-\alpha} \frac{1}{c^k} e^{\pi i\alpha} i^{2s} \delta_{\kappa_\gamma,0}b_\gamma(0) \int^{i\infty}_{i/t} z^{k-s-2}dz.
\end{eqnarray*}
The last integral can be computed as
\[
\int^{i\infty}_{i/t} z^{k-s-2}dz = -\frac{i^{k-s-1}t^{s-k+1}}{(k-s-1)}.
\]
Therefore,  we have
\begin{eqnarray} \label{computation2}
\nonumber &&\int^{\reg}_{[z_1,r]} g(\tau)(\tau-r)^{s}d\tau\\
 &&= \frac{1}{c^k} (c^2)^{k-s-1} i^{k+s-1} \sum_{n\gg-\infty\atop n+\kappa
_\gamma\neq 0} \frac{b_\gamma(n)e^{2\pi i\left( \frac{n+\kappa_\gamma}{\lambda_\gamma}\right) \left(-\frac dc\right)}}{\left(2\pi \left(\frac{n+\kappa_\gamma}{\lambda_\gamma}\right)\right)^{k-s-1}}
 \Gamma\left(k-s-1, \frac{2\pi}{c^2t}\left(\frac{n+\kappa_\gamma}{\lambda_\gamma}\right)\right)
 \\
\nonumber && - \delta_{\kappa_\gamma,0}   \frac{ i^{k+s-1}t^{s-k+1} b_\gamma(0) }{c^k(k-s-1)}.
\end{eqnarray}
As in the proof of Lemma \ref{generalpoint}, we can see that the series in (\ref{computation2}) converges absolutely for any $s\in\CC$ except $s = k-1$ by the asymptotic properties of the incomplete gamma function and Fourier coefficients of $g$, i.e.,
\[
b_\gamma(n) = O\left( e^{C\sqrt{|n|}}\right),\ n\to\infty
\]
for a constant $C>0$.
Hence, this series gives the desired meromorphic continuation.
\end{proof}

The following lemma shows that the regularized integral defined in (\ref{twocusp}) is independent of the choice of $z_1$.

\begin{lem} \label{independent}
Let $z_0, z_1\in\HH$. Then, we have
\begin{eqnarray*}
&&-\int^{\reg}_{[z_1,r_1]} g(\tau)(\tau-\overline{z_2})^{s}d\tau + \int^{\reg}_{[z_1,r_2]} g(\tau)(\tau-\overline{z_2})^{s}d\tau\\
&&= -\int^{\reg}_{[z_0,r_1]} g(\tau)(\tau-\overline{z_2})^{s}d\tau + \int^{\reg}_{[z_0,r_2]} g(\tau)(\tau-\overline{z_2})^{s}d\tau,
\end{eqnarray*}
if each integral is well defined.
\end{lem}

\begin{proof} [\bf Proof of Lemma \ref{independent}]
Let us compute the following sum of the integrals
\begin{eqnarray} \label{original}
&&-\int^{\reg}_{[z_1,r_1]} g(\tau)(\tau-\overline{z_2})^{s}d\tau + \int^{\reg}_{[z_1,r_2]} g(\tau)(\tau-\overline{z_2})^{s}d\tau\\
\nonumber &&+\int^{\reg}_{[z_0,r_1]} g(\tau)(\tau-\overline{z_2})^{s}d\tau - \int^{\reg}_{[z_0,r_2]} g(\tau)(\tau-\overline{z_2})^{s}d\tau.
\end{eqnarray}
From the definition of the regularized Eichler integral in (\ref{othercusp}), we see that (\ref{original}) is equal to
\begin{eqnarray*}
&&-\left[ \left[ \int^{i\infty}_{\gamma_1^{-1}(z_1)} (c_1\tau+d_1)^{-2+\alpha}g(\gamma_1\tau)(\gamma_1\tau-\overline{z_2})^{s+\alpha}e^{2\pi i\beta\tau}d\tau \right]_{\beta=0} \right]_{\alpha=0}\\
&&+\left[ \left[ \int^{i\infty}_{\gamma_1^{-1}(z_0)} (c_1\tau+d_1)^{-2+\alpha}g(\gamma_1\tau)(\gamma\tau-\overline{z_2})^{s+\alpha}e^{2\pi i\beta\tau}d\tau \right]_{\beta=0} \right]_{\alpha=0}\\
&&+\left[ \left[ \int^{i\infty}_{\gamma_2^{-1}(z_1)} (c_2\tau+d_2)^{-2+\alpha}g(\gamma_2\tau)(\gamma_2\tau-\overline{z_2})^{s+\alpha}e^{2\pi i\beta\tau}d\tau \right]_{\beta=0} \right]_{\alpha=0}\\
&&-\left[ \left[ \int^{i\infty}_{\gamma_2^{-1}(z_0)} (c_2\tau+d_2)^{-2+\alpha}g(\gamma_2\tau)(\gamma_2\tau-\overline{z_2})^{s+\alpha}e^{2\pi i\beta\tau}d\tau \right]_{\beta=0} \right]_{\alpha=0}\\
&&=\left[ \left[ \int^{\gamma_1^{-1}(z_1)}_{\gamma_1^{-1}(z_0)} (c_1\tau+d_1)^{-2+\alpha}g(\gamma_1\tau)(\gamma_1\tau-\overline{z_2})^{s+\alpha}e^{2\pi i\beta\tau}d\tau \right]_{\beta=0} \right]_{\alpha=0}\\
&&+\left[ \left[ \int^{\gamma_2^{-1}(z_0)}_{\gamma_2^{-1}(z_1)} (c_2\tau+d_2)^{-2+\alpha}g(\gamma_2\tau)(\gamma_2\tau-\overline{z_2})^{s+\alpha}e^{2\pi i\beta\tau}d\tau \right]_{\beta=0} \right]_{\alpha=0},
\end{eqnarray*}
where $\gamma_i = \sm a_i&b_i\\c_i&d_i\esm\in \SL_2(\ZZ)$ satisfies $r_i = \gamma_i(i\infty)$ for $i=1, 2$.
Since
\[
(c_i\tau+d_i)^{-2}g(\gamma_i\tau)(\gamma_i\tau-\overline{z_2})^{s}
\]
is holomorphic on $\HH$ as a function of $\tau$ for $i=1,2$, the last sum is equal to
\begin{eqnarray} \label{integralcompt}
&&\int^{\gamma_1^{-1}(z_1)}_{\gamma_1^{-1}(z_0)} (c_1\tau+d_1)^{-2}g(\gamma_1\tau)(\gamma_1\tau-\overline{z_2})^{s}d\tau +\int^{\gamma_2^{-1}(z_0)}_{\gamma_2^{-1}(z_1)} (c_2\tau+d_2)^{-2}g(\gamma_2\tau)(\gamma_2\tau-\overline{z_2})^{s}d\tau.
\end{eqnarray}
By applying the change of variables $w = \gamma_i\tau$ for $i=1,2$,  the sum in (\ref{integralcompt}) takes the form
\begin{eqnarray*}
\int^{z_1}_{z_0} g(w)(w-\overline{z_2})^{s}dw + \int^{z_0}_{z_1} g(w)(w-\overline{z_2})^{s}dw = 0.
\end{eqnarray*}
This completes the proof.
\end{proof}

\subsection{Properties of regularized Eichler integrals}
The following lemma shows that the regularized Eichler integral (\ref{twocusp}) satisfies certain transformation properties.
\begin{lem} \label{transform}
Let $g\in M^!_{k}(\Gamma_0(4N))$ and $r_1, r_2\in \QQ\cup \{i\infty\}$.
Suppose that $\gamma = \sm a&b\\c&d\esm\in \Gamma_0(4N)$.
Then, we have
\[
\int^{\reg}_{[r_1, r_2]} g(\tau)(\tau-\overline{z})^{k-2}d\tau
= \chi^{2k}_{\theta}(\gamma) (-c\overline{z}+a)^{k-2}\int^{\reg}_{[\gamma^{-1}(r_1), \gamma^{-1}(r_2)]} g(\tau)(\tau-\overline{\gamma^{-1}z})^{k-2}d\tau
\]
for $z\in\HH$
\end{lem}

\begin{proof}  [\bf Proof of Lemma \ref{transform}]
By the definition of the regularized Eichler integral, we have
\begin{equation*}
\int^{\reg}_{[r_1,r_2]} g(\tau)(\tau-\overline{z})^{k-2}d\tau
= -\int^{\reg}_{[z_1,r_1]} g(\tau)(\tau-\overline{z})^{k-2}d\tau
+\int^{\reg}_{[z_1,r_2]} g(\tau)(\tau-\overline{z})^{k-2}d\tau
\end{equation*}
for $z_1\in \HH$.
By Lemma \ref{welldefinedforcusp}, we obtain
\begin{eqnarray} \label{transform1}
\int^{\reg}_{[r_1,r_2]} g(\tau)(\tau-\overline{z})^{k-2}d\tau&=&-\chi^{2k}_{\theta}(\gamma_1)j(\gamma_1^{-1}, \overline{z})^{k-2}\int^{\reg}_{[\gamma_1^{-1}z_1,i\infty]} g(\tau)(\tau-\overline{\gamma_1^{-1}z})^{k-2}d\tau\\
\nonumber &&+ \chi^{2k}_{\theta}(\gamma_2)j(\gamma_2^{-1}, \overline{z})^{k-2}\int^{\reg}_{[\gamma_2^{-1}z_1,i\infty]} g(\tau)(\tau-\overline{\gamma_2^{-1}z})^{k-2}d\tau,
\end{eqnarray}
where $\gamma_i \in \SL_2(\ZZ)$ such that $\gamma_i(i\infty) = r_i$ for $i=1,2$. Here, $j(\gamma, z) = cz+d$ for $\sm a&b\\c&d\esm\in \SL_2(\ZZ)$.
If we use (\ref{transform1}), we have
\begin{eqnarray} \label{transform2}
&&\chi^{2k}_{\theta}(\gamma) j(\gamma^{-1},\overline{z})^{k-2}\int^{\reg}_{[\gamma^{-1}(r_1), \gamma^{-1}(r_2)]} g(\tau)(\tau-\overline{\gamma^{-1}z})^{k-2}d\tau\\
\nonumber &&= -\chi^{2k}_{\theta}(\gamma) j(\gamma^{-1},\overline{z})^{k-2} \chi^{2k}_{\theta}(\gamma^{-1}\gamma_1)j(\gamma_1^{-1}\gamma, \overline{\gamma^{-1} z})^{k-2}\int^{\reg}_{[\gamma_1^{-1}\gamma z_2,i\infty]} g(\tau)(\tau-\overline{\gamma_1^{-1}z})^{k-2}d\tau\\
\nonumber &&+ \chi^{2k}_{\theta}(\gamma) j(\gamma^{-1},\overline{z})^{k-2} \chi^{2k}_{\theta}(\gamma^{-1}\gamma_2)j(\gamma_2^{-1}\gamma, \overline{\gamma^{-1} z})^{k-2}\int^{\reg}_{[\gamma_2^{-1}\gamma z_2,i\infty]} g(\tau)(\tau-\overline{\gamma_2^{-1}z})^{k-2}d\tau
\end{eqnarray}
for $z_2\in\HH$

Note that by the consistency condition of the multiplier system $\chi_{\theta}^{2k}$ we have
\[
\chi^{2k}_{\theta}(\gamma_i^{-1}\gamma) j(\gamma_i^{-1}\gamma, w)^{k-2} = \chi^{2k}_{\theta}(\gamma_i^{-1}) j(\gamma_i^{-1}, \gamma w)^{k-2} \chi^{2k}_{\theta}(\gamma) j(\gamma,w)^{k-2}
\]
for $w\in \HH$ and $i = 1, 2$. If we take $w = \gamma^{-1}z$, then we obtain
\[
\chi^{2k}_{\theta}(\gamma_i^{-1}\gamma) j(\gamma_i^{-1}\gamma, \gamma^{-1}z)^{k-2} = \chi^{2k}_{\theta}(\gamma_i^{-1}) j(\gamma_i^{-1}, z)^{k-2} \chi^{2k}_{\theta}(\gamma) j(\gamma,\gamma^{-1}z)^{k-2}
\]
for $i=1,2$. Since $j(\gamma, \gamma^{-1}z)^{k-2} = j(\gamma^{-1}, z)^{-(k-2)}$, we have
\[
\chi^{-2k}_{\theta}(\gamma) j(\gamma^{-1},z)^{k-2} \chi^{-2k}_{\theta}(\gamma^{-1}\gamma_i) j(\gamma_i^{-1}\gamma, \gamma^{-1}z)^{k-2} = \chi^{-2k}_{\theta}(\gamma_i) j(\gamma_i^{-1}, z)^{k-2}
\]
for $i = 1, 2$.
This implies that (\ref{transform2}) is equal to
\begin{eqnarray} \label{transform3}
&&-\chi^{2k}_{\theta}(\gamma_1) j(\gamma_1^{-1}, \overline{z})^{k-2} \int^{\reg}_{[\gamma_1^{-1}\gamma z_2, i\infty]} g(\tau)(\tau-\overline{\gamma_1^{-1}z})^{k-2}d\tau\\
\nonumber &&+\chi^{2k}_{\theta}(\gamma_2) j(\gamma_2^{-1}, \overline{z})^{k-2}\int^{\reg}_{[\gamma_2^{-1}\gamma z_2, i\infty]} g(\tau)(\tau-\overline{\gamma_2^{-1}z})^{k-2}d\tau.
\end{eqnarray}
If we take $z_2 = \gamma^{-1}z_1$, then by (\ref{transform1})  we see that (\ref{transform3}) is equal to
$
\int^{\reg}_{[r_1,r_2]} g(\tau)(\tau-\overline{z})^{k-2}d\tau.
$
\end{proof}

The following lemma shows that we can do some change of variables in the regularized Eichler integrals as we do in the usual integrals.

\begin{lem} \label{Mtimes}
Let $f\in M^!_{k}(\Gamma_0(4N))$ and $g(z) := f(M^2z)$ for $M\in\NN$.
Let $\gamma = \sm a&b\\c&d\esm\in \Gamma_0(4NM^2)$.
Then, we have
\[
\int^{\reg}_{[\gamma(i\infty), i\infty]} g(\tau)(\tau-\overline{z})^{k-2}d\tau = \frac{1}{M^{2k-2}} \int^{\reg}_{[\gamma'(i\infty), i\infty]} f(\tau)(\tau- \overline{M^2z})^{k-2}d\tau
\]
for $z\in \HH$,
where $\gamma' = \sm a& M^2b\\ c/M^2 & d\esm\in \Gamma_0(4N)$.
\end{lem}

\begin{proof} [\bf Proof of Lemma \ref{Mtimes}]
Let $z_1\in \HH$ be fixed.
Note that we have
\begin{eqnarray*}
\int^{\reg}_{[z_1, i\infty]} f(M^2\tau)(\tau-\overline{z})^{k-2}d\tau
&=&\left[ \left[ \int^{i\infty}_{z_1} f(M^2\tau)(\tau-\overline{z})^{k-2+\alpha} e^{2\pi i\beta\tau}d\tau \right]_{\beta=0} \right]_{\alpha=0}\\
&=&\frac{1}{M^2} \left[ \left[ \int^{i\infty}_{M^2z_1} f(\tau)\left(\frac{\tau}{M^2}-\overline{z}\right)^{k-2+\alpha} e^{2\pi i\beta\tau/M^2}d\tau \right]_{\beta=0} \right]_{\alpha=0}\\
&=&\frac{1}{M^{2k-2}} \left[ \left[ \int^{i\infty}_{M^2z_1} f(\tau)\left(\tau-\overline{M^2z}\right)^{k-2+\alpha} e^{2\pi i\beta\tau/M^2}d\tau \right]_{\beta=0} \right]_{\alpha=0}\\
&=&\frac{1}{M^{2k-2}} \int^{\reg}_{[M^2 z_1, i\infty]} f(\tau)(\tau-\overline{M^2z})^{k-2}d\tau.
\end{eqnarray*}
Therefore, by Lemma \ref{welldefinedforcusp}, we obtain
\begin{eqnarray} \label{Mtimes2}
&&\int^{\reg}_{[\gamma(i\infty), i\infty]} g(\tau)(\tau-\overline{z})^{k-2}d\tau\\
\nonumber &&= -\int^{\reg}_{[z_1,\gamma(i\infty)]} g(\tau)(\tau-\overline{z})^{k-2}d\tau + \int^{\reg}_{[z_1,i\infty]} g(\tau)(\tau-\overline{z})^{k-2}d\tau\\
\nonumber &&= -\chi^{2k}_{\theta}(\gamma) j(\gamma^{-1}, \overline{z})^{k-2}\int^{\reg}_{[\gamma^{-1}z_1, i\infty]} g(\tau)(\tau-\overline{\gamma^{-1}z})^{k-2}d\tau + \int^{\reg}_{[z_1, i\infty]} g(\tau)(\tau-\overline{z})^{k-2}d\tau\\
\nonumber && = -\chi^{2k}_{\theta}(\gamma)j(\gamma^{-1}, \overline{z})^{k-2}\frac{1}{M^{2k-2}} \int^{\reg}_{[M^2(\gamma^{-1}z_1), i\infty]} f(\tau)(\tau-\overline{M^2(\gamma^{-1}z)})^{k-2}d\tau\\
\nonumber && + \frac{1}{M^{2k-2}} \int^{\reg}_{[M^2z_1, i\infty]} f(\tau)(\tau-\overline{M^2z})^{k-2}d\tau.
\end{eqnarray}
Note that we have
\begin{eqnarray*}
M^2(\gamma^{-1}z) &=& \gamma'^{-1}(M^2 z), \chi_{\theta}(\gamma) = \chi_{\theta}(\gamma'),\ \text{\and}\ j(\gamma^{-1}, z) = j(\gamma'^{-1}, M^2 z)
\end{eqnarray*}
for $z\in\HH$.
Therefore, (\ref{Mtimes2}) is equal to
\begin{eqnarray*}
&&-\frac{1}{M^{2k-2}} \chi^{2k}_{\theta}(\gamma')j(\gamma'^{-1}, \overline{M^2z})^{k-2}\int^{\reg}_{[\gamma'^{-1}(M^2 z_1), i\infty]} f(\tau) (\tau-\overline{\gamma'^{-1}(M^2 z)})^{k-2}d\tau\\
&& + \frac{1}{M^{2k-2}} \int^{\reg}_{[M^2z_1, i\infty]} f(\tau)(\tau-\overline{M^2z})^{k-2}d\tau\\
&& = \frac{1}{M^{2k-2}} \int^{\reg}_{[\gamma'(i\infty), i\infty]} f(\tau)(\tau-\overline{M^2z})^{k-2}d\tau.
\end{eqnarray*}
\end{proof}

The formula (\ref{nonholomorphicpart}) applies to compute the period function of the nonholomorphic part of $f$ in terms of regularized Eichler integrals associated with $g$.

\begin{lem} \label{periodfunction}
Let $f\in H_{2-k}(\Gamma_0(4N))$ and $g = \xi_{2-k}(f) \in M^!_{k}(\Gamma_0(4N))$.
Then, we have
\[
(f^- - f^-|_{2-k}\gamma)(z) = -\frac{1}{(-2i)^{k-1}} \overline{\int^{\reg}_{[\gamma^{-1}
(i\infty),i\infty]} g(\tau)(\tau-\bar{z})^{k-2}d\tau}
\]
for $\gamma\in \Gamma_0(4N)$ and $z\in \HH$.
\end{lem}

\begin{proof} [\bf Proof of Lemma \ref{periodfunction}]
Let $\gamma = \sm a&b\\c&d\esm\in \Gamma_0(4N)$. We may assume that $c>0$ since we have $f^-|_{2-k}\gamma = f^-|_{2-k}
(-\gamma)$.
By (\ref{nonholomorphicpart}), we obtain
\begin{eqnarray*}
(f^-|_{2-k}\gamma)(z) &=& \chi_{\theta}^{2k}(\gamma) (cz+d)^{k-2}f^-(\gamma z)= \frac{-\chi_{\theta}^{2k}(\gamma)(cz+d)^{k-2}}{(-2i)^{k-1}}\overline{\int^{\reg}_{[\gamma z,i\infty]} g(\tau)(\tau-\overline{\gamma z})^{k-2}d\tau}.
\end{eqnarray*}
By Lemma \ref{welldefinedforcusp}, we have
\begin{equation} \label{slashpart}
\chi^{-2k}_{\theta}(\gamma)(c\overline{z}+d)^{k-2}\int^{\reg}_{[\gamma z, i\infty]} g(z)(\tau-\overline{\gamma z})^{k-2}d\tau = \int^{\reg}_{[z, \gamma^{-1}
(i\infty)]} g(\tau)(\tau-\overline{z})^{k-2}d\tau.
\end{equation}
Therefore, we obtain
\[
(f^-|_{2-k}\gamma)(z) = -\frac{1}{(-2i)^{k-1}} \overline{\int^{\reg}_{[z,\gamma^{-1}
(i\infty)]} g(\tau)(\tau-\overline{z})^{k-2}d\tau}.\]
If we combine (\ref{nonholomorphicpart}) and (\ref{slashpart}),  we have
\begin{eqnarray*}
(f^- - f^-|_{2-k}\gamma)(z)  &=& -\frac{1}{(-2i)^{k-1}}\left[ \overline{\int^{\reg}_{[z,i\infty]} g(\tau)(\tau-\bar{z})^{k-2}d\tau} - \overline{\int^{\reg}_{[z,\gamma^{-1}
(i\infty)]} g(\tau)(\tau-\overline{z})^{k-2}d\tau}\right]\\
&=&-\frac{1}{(-2i)^{k-1}} \overline{\int^{\reg}_{[\gamma^{-1}
(i\infty),i\infty]} g(\tau)(\tau-\bar{z})^{k-2}d\tau},
\end{eqnarray*}
which completes the proof.
\end{proof}

\begin{rmk}
The proof of Theorem \ref{maintheoremsection3} is given by Lemma \ref{generalpoint}, \ref{welldefinedforcusp}, and \ref{periodfunction} and Remark \ref{nonholo} (5).
\end{rmk}

\section{Regularized $L$-functions}\label{Regularized $L$-function}

This section is devoted to studying the regularized $L$-function of a weakly holomorphic modular form. Especially, we prove that the $L$-function is meromorphic on $\mathbb{C}$ and satisfies a certain functional equation.

Let $g$ be a weakly holomorphic modular form in $M_{k}^{!}(\Gamma_0(4N))$.
Assume that, for each $\gamma = \sm a&b\\c&d\esm\in\SL_2(\ZZ)$, the function $(cz+d)^{-k}g(\gamma z)$ has a Fourier expansion as in (\ref{fourierq}).
For $r\in\QQ$, we define a {\it  regularized twisted $L$-function of $g$ associated with $r$} by
\begin{eqnarray*}
\nonumber L^{\reg}_r(g,s) &:=& \frac{1}{\Gamma(s)}\sum_{n\gg-\infty\atop n\neq0} \frac{b(n)e^{2\pi inr}}{n^s} \Gamma(s, 2\pi n) \\
&&+ \frac{i^{k}}{c^k\Gamma(s)} (2\pi)^{2s-k} (c^2)^{k-s}
\sum_{n\gg-\infty\atop n+\kappa_{\gamma}\neq0} \frac{b_\gamma(n)e^{2\pi i(n+\kappa_\gamma)(-\frac{d}{c})/\lambda_\gamma}}{\left(\frac{n+\kappa_\gamma}{\lambda_\gamma}\right)^{k-s}} \Gamma\left(k-s, \frac{2\pi(n+\kappa_\gamma)}{c^2\lambda_\gamma}\right)\\
\nonumber&& -\frac{(2\pi)^{s} b(0)}{s\Gamma(s)}- \delta_{\kappa_\gamma,0} \frac{i^k (2\pi)^s b_\gamma(0)}{c^k (k-s)\Gamma(s)},
\end{eqnarray*}
where $\delta_{\kappa_\gamma,0}$ is the Kronecker delta, and $\gamma = \sm a&b\\c&d\esm\in\SL_2(\ZZ)$  such that
$r= \gamma(i\infty)$. 
 When the constant term of $g$ is zero, this $L$-function was defined in \cite{F}.
This definition is independent of the choice of $\gamma$ (see Lemma \ref{Lseries} (2)).

\begin{lem} \label{Lseries}
Let $g$ be a weakly holomorphic modular form in $M_{k}^{!}(\Gamma_0(4N))$.
\begin{enumerate}
\item For each $r\in \QQ$, the regularized $L$-function $L^{\reg}_{r}(g,s)$ converges absolutely on $\CC$ except $s = 0$ and $s = k$.
If $g\in S^!_{k}(\Gamma_0(4N))$, then $L^{\reg}_{r}(g,s)$ is entire as a function of $s$, where $S^!_{k}(\Gamma_0(4N))$ denotes the subspace of $M^!_{k}(\Gamma_0(4N))$ consisting of $f$ whose constant terms vanish at all cusps of $\Gamma_0(4N)$.

\item The regularized $L$-function $L^{\reg}_{r}(g,s)$ is the meromorphic continuation of the integral
\begin{equation} \label{integralrepresentation}
\frac{1}{\Gamma(s)}\left(\frac{2\pi}{i}\right)^s\int^{\reg}_{[r,i\infty]} g(\tau)(\tau-r)^{s-1}d\tau,
\end{equation}
which is well defined for $\re(s)>k$.
If $g \in S^!_{k}(\Gamma_0(4N))$, then the integral in (\ref{integralrepresentation}) is entire as a function of $s$.
In particular, the definition of $L^{\reg}_r(g,s)$ is independent of the choice of $\gamma$.

\item Let $\gamma = \sm a&b \\ c&d\esm\in \Gamma_0(4N)$. Then, the regularized $L$-function $L^{\reg}_{r}(g,s)$ satisfies the functional equation
\[
\Gamma(s)(2\pi)^{-s} L^{\reg}_{\gamma(i\infty)}(g,s)
= i^k c^k (c^2)^{-s}\chi_{\theta}^{2k}(\gamma) \Gamma(k-s)(2\pi)^{-(k-s)} L^{\reg}_{\gamma^{-1}(i\infty)}(g,k-s).
\]
\end{enumerate}
\end{lem}

\begin{proof} [\bf Proof of Lemma \ref{Lseries}]
(1) Note that $\Gamma(s,x), b(n)$ and $b_\gamma(n)$ have the following
  asymptotic behaviors
\[
\Gamma(s,x) \sim |x|^{s-1}e^{-x},\ x\to \infty
\]
and
\[
b(n) = O\left( e^{C\sqrt{|n|}}\right),\ b_\gamma(n) = O\left( e^{C\sqrt{|n|}}\right),\ n\to\infty
\]
for a constant $C>0$. Therefore, the regularized $L$-function $L^{\reg}_{r}(g,s)$ converges absolutely on $\CC$ except at $s = 0$ and $s = k$.

If $g\in S^!_{k}(\Gamma_0(4N))$, then $b(0) = \delta_{\kappa_\gamma, 0}b_\gamma(0) = 0$.
Let $f_n(s) = \frac{b(n)e^{2\pi inr}}{n^s} \Gamma(s, 2\pi n)$ for $n\in\NN$.
Let $A\subseteq \CC$ be bounded.
Then, we obtain
\[
f'_n(s) =  \frac{b(n)e^{2\pi inr}}{(\mathrm{ln}n) n^s} \Gamma(s, 2\pi n)
+ \frac{b(n)e^{2\pi inr}}{n^s} (s-1)\Gamma(s-1, 2\pi n).
\]
By the asymptotic behavior of the incomplete gamma function, for a given $\epsilon>0$, there exists $N\in\NN$ such that if $n\geq N$ and $s\in A$, then
\[
|\Gamma(s, 2 \pi n)| <(1+\epsilon) (2\pi n)^{s-1} e^{-2\pi n}.
\]
Moreover, since we have
\[
\lim_{n\to\infty} \left(  \frac{b(n)e^{2\pi inr}}{(\mathrm{ln}n)n^s} (2\pi n)^{s-1} + \frac{b(n)e^{2\pi inr}}{n^s} (s-1) (2\pi n)^{s-2}\right) \frac{1}{e^{n}} = 0,
\]
there exists $N'\in \NN$ such that if $n\geq N'$ and $s\in A$, then
\[
\left|  \frac{b(n)e^{2\pi inr}}{(\mathrm{ln}n)n^s} (2\pi n)^{s-1} + \frac{b(n)e^{2\pi inr}}{n^s} (s-1) (2\pi n)^{s-2}\right| <  e^{n}.
\]
Therefore, if $n\geq \mathrm{max}\{N, N'\}$ and $s\in A$, then we have
\[
|f_n'(s)| < (1+\epsilon) e^{(1-2\pi)n}.
\]
This implies that the series $\displaystyle\sum_{n\gg-\infty\atop n\neq 0} f_n'(s)$ converges uniformly on $A$. Therefore, the series
\[
\frac{1}{\Gamma(s)}\sum_{n\gg-\infty\atop n\neq0} \frac{b(n)e^{2\pi inr}}{n^s} \Gamma(s, 2\pi n)
\]
is entire as a function of $s$.
In the same way, we can see that the series
\[
\frac{i^{k}}{\Gamma(s)}\left(\frac{2\pi}{c}\right)^{2s-k}\sum_{n\gg-\infty\atop n+\kappa_{\gamma}\neq0} \frac{b_\gamma(n)e^{2\pi i(n+\kappa_\gamma)(-\frac{d}{c})/\lambda_\gamma}}{\left(\frac{n+\kappa_\gamma}{\lambda_\gamma}\right)^{k-s}} \Gamma\left(k-s, \frac{2\pi(n+\kappa_\gamma)}{c^2\lambda_\gamma}\right)
\]
is also entire as a function of $s$. Hence, $L^{\reg}_{r}(g,s)$ is entire as a function of $s$.

(2) Suppose that $\re(s)>k$.
By the definition of the regularized Eichler integral in (\ref{twocusp}), we have
\begin{equation} \label{decomposition}
\int^{\reg}_{[r,i\infty]} g(\tau)(\tau-r)^{s-1}d\tau = \int^{\reg}_{[z_1,i\infty]}g(\tau)(\tau-r)^{s-1}d\tau - \int^{\reg}_{[z_1,r]} g(\tau)(\tau-r)^{s-1}d\tau
\end{equation}
for any $z_1\in \HH$.
Let $z_1 = r + it$ for $t>0$.
Then, Lemma \ref{generalpoint} and (\ref{computation2}) implies
\begin{equation} \label{lseries1}
 \int^{\reg}_{[z_1,i\infty]}g(\tau)(\tau-r)^{s-1}d\tau=i^s\sum_{n\gg-\infty\atop n\neq0} \frac{b(n)e^{2\pi inr}}{(2\pi n)^s} \Gamma(s, 2\pi nt)
-i^st^s\frac{b(0)}{s}
\end{equation}
and
\begin{eqnarray} \label{lseries2}
\nonumber \int^{\reg}_{[z_1,r]} g(\tau)(\tau-r)^{s-1}d\tau&=& \frac{1}{c^k}(c^2)^{k-s}i^{k+s-2} \sum_{n\gg-\infty\atop n+\kappa
_\gamma\neq 0} \frac{b_\gamma(n)e^{2\pi i\left( \frac{n+\kappa_\gamma}{\lambda_\gamma}\right) \left(-\frac dc\right)}}{\left(2\pi \left(\frac{n+\kappa_\gamma}{\lambda_\gamma}\right)\right)^{k-s}}
 \Gamma\left(k-s, \frac{2\pi}{c^2t}\left(\frac{n+\kappa_\gamma}{\lambda_\gamma}\right)\right)
 \\
&& - \delta_{\kappa_\gamma,0}  \frac{i^{k+s-2}t^{s-k} b_\gamma(0) }{c^k (k-s)}.
\end{eqnarray}

If we combine (\ref{lseries1}) and (\ref{lseries2}), then we obtain
\begin{eqnarray} \label{tindependent}
\nonumber &&\frac{1}{\Gamma(s)}\left(\frac{2\pi}{i}\right)^s\int^{\reg}_{[r,i\infty]} g(\tau)(\tau-r)^{s-1}d\tau\\
 &&= \frac{1}{\Gamma(s)}\sum_{n\gg-\infty\atop n\neq0} \frac{b(n)e^{2\pi inr}}{n^s} \Gamma(s, 2\pi nt) \\
\nonumber &&+ \frac{i^{k}}{c^k\Gamma(s)} (2\pi)^{2s-k} (c^2)^{k-s}
\sum_{n\gg-\infty\atop n+\kappa_{\gamma}\neq0} \frac{b_\gamma(n)e^{2\pi i(n+\kappa_\gamma)(-\frac{d}{c})/\lambda_\gamma}}{\left(\frac{n+\kappa_\gamma}{\lambda_\gamma}\right)^{k-s}} \Gamma\left(k-s, \frac{2\pi(n+\kappa_\gamma)}{c^2t\lambda_\gamma}\right)\\
\nonumber&& -\frac{(2\pi)^{s}t^s b(0)}{s\Gamma(s)}- \delta_{\kappa_\gamma,0} \frac{i^k (2\pi)^s t^{s-k}b_\gamma(0)}{ c^k(k-s)\Gamma(s)},
\end{eqnarray}
for any $t>0$. If we let $t=1$, then we see that
\begin{equation} \label{lseries3}
\frac{1}{\Gamma(s)}\left(\frac{2\pi}{i}\right)^s\int^{\reg}_{[r,i\infty]} g(\tau)(\tau-r)^{s-1}d\tau = L^{\reg}_{r}(g,s)
\end{equation}
for $s>k$. Therefore, the integral $\frac{1}{\Gamma(s)}\left(\frac{2\pi}{i}\right)^s\int^{\reg}_{[r,i\infty]} g(\tau)(\tau-r)^{s-1}d\tau$ can be continued meromorphically to all $s\in\CC$ and is equal to
  the regularized $L$-function $L^{\reg}_{r}(g,s)$.

In the above computation, we need the condition $\re(s)>k$ only for the regularized integral $\int^{\reg}_{[z_1,r]} g(\tau)(\tau-r)^{s-1}d\tau$ in (\ref{decomposition}).
More precisely, we need the condition $\re(s)>k$ to compute the term
\[
\delta_{\kappa_\gamma,0}b_\gamma(0) \int^{i\infty}_{i/t} z^{k-s-1}dz
\]
(see the proof of Lemma \ref{welldefinedforcusp}).
If $g\in S^!_{k}(\Gamma_0(4N))$, then
$\delta_{\kappa_\gamma,0}b_\gamma(0) = 0$, and hence the equality (\ref{lseries3}) holds for all $s$.
Since $L_r^{\reg}(g,s)$ is entire as a function of $s$ for $g\in S^!_{k}(\Gamma_0(4N))$ by (1), the integral in (\ref{integralrepresentation}) gives an entire function  of $s$.

(3) By Lemma \ref{generalpoint} and (\ref{computation2}), we see that
\[
\int^{\reg}_{[z_1,r]}g(\tau)(\tau-r)^{s-1}d\tau = \frac{1}{c^k}(c^2)^{k-s}i^{2s-2}\chi^{2k}_{\theta}(\gamma)\int^{\reg}_{[\gamma^{-1}z_1, i\infty]} g(\tau)\left(\tau-\left(-\frac dc\right)\right)^{k-s-1}d\tau
\]
for $z_1 = r+it\in \HH$,
and the integral in the right hand side is well defined for all $s\in \CC$ except at $s = k$.
Therefore, we have
\begin{eqnarray} \label{ftnaleqn}
&&\Gamma(s)\left(\frac{2\pi}{i}\right)^{-s} L^{\reg}_{\gamma(i\infty)}(g,s) = \int^{\reg}_{[\gamma(i\infty), i\infty]} g(\tau)(\tau-\gamma(i\infty))^{s-1}d\tau\\
\nonumber &&= \int^{\reg}_{[z_1, i\infty]} g(\tau)(\tau-\gamma(i\infty))^{s-1}d\tau
- \int^{\reg}_{[z_1, \gamma(i\infty)]} g(\tau)(\tau-\gamma(i\infty))^{s-1}d\tau\\
\nonumber &&= \int^{\reg}_{[z_1, i\infty]} g(\tau)(\tau-\gamma(i\infty))^{s-1}d\tau - \frac{1}{c^k}(c^2)^{k-s}i^{2s-2}\chi^{2k}_{\theta}(\gamma)\int^{\reg}_{[\gamma^{-1}z_1, i\infty]} g(\tau)(\tau-\gamma^{-1}(i\infty))^{k-s-1}d\tau
\end{eqnarray}
for $z_1 = \gamma(i\infty) + it\in\HH$. Here, we can take any positive real number $t>0$.
Note that (\ref{ftnaleqn}) is equal to
\begin{eqnarray*}
&&- \frac{1}{c^k}(c^2)^{k-s}i^{2s-2}\chi^{2k}_{\theta}(\gamma)\biggl[ \int^{\reg}_{[\gamma^{-1}z_1, i\infty]} g(\tau)(\tau-\gamma^{-1}(i\infty))^{k-s-1}d\tau\\
&& \qquad - c^{k}(c^{2})^{s-k} i^{-(2s-2)}\chi^{-2k}_{\theta}(\gamma)\int^{\reg}_{[z_1, i\infty]} g(\tau)(\tau-\gamma(i\infty))^{s-1}d\tau\biggr]\\
&&= - \frac{1}{c^k}(c^2)^{k-s}i^{2s-2}\chi^{2k}_{\theta}(\gamma)\biggl[ \int^{\reg}_{[\gamma^{-1}z_1, i\infty]} g(\tau)(\tau-\gamma^{-1}(i\infty))^{k-s-1}d\tau\\
&& \qquad - \frac{1}{(-c)^k} (c^2)^{s}i^{2(k-s)-2}\chi^{2k}_{\theta}(\gamma^{-1})
\int^{\reg}_{[z_1, i\infty]} g(\tau)(\tau-\gamma(i\infty))^{s-1}d\tau\biggr]\\
&&= - \frac{1}{c^k}(c^2)^{k-s}i^{2s-2}\chi^{2k}_{\theta}(\gamma) \Gamma(k-s)\left(\frac{2\pi}{i}\right)^{-(k-s)}L^{\reg}_{\gamma^{-1}(i\infty)}(g,k-s).
\end{eqnarray*}
Here, we used the identity $(c^2)^{-k}c^{k} = \frac{1}{(-c)^k}i^{2k}$ for any nonzero real number $c$.
From this, we can derive the desired functional equation for $L^{\reg}_{r}(g,s)$.
\end{proof}

The following theorem shows that, for $g\in H_{2-k}(\Gamma_0(4N))$, the limit of $g^+$  can be described in terms of the special value of the regularized $L$-function of $\xi_{2-k}(g)$.

\begin{thm} \label{mocklimit}
Let $g\in H_{2-k}(\Gamma_0(4N))$.
Suppose that $g^+$ does not have neither principal part nor constant term in the Fourier expansion at the cusp $i\infty$.
Let $\gamma = \sm a&b\\c&d\esm \in \Gamma_0(4N)$ with $c>0$ and $r = \gamma(i\infty)$.
Then for any $x\in\RR$ we have
\begin{align} \label{mockquantum}
 &\lim_{t\to 0} \left[g^+\left(\gamma\left(x+\frac{1}{c^2t}i\right)\right) -\frac{\chi^{-2k}_{\theta}(\gamma)\overline{b(0)}}{2^{k-1}(k-1)} \left( c\left(x+\frac{1}{c^2t}i\right)+d\right)^{2-k}\left(\left(\frac{1}{c^2t}+\frac{1}{c}\right)-i\left(x+\frac dc\right)\right)^{k-1}\right]\\
\nonumber &\qquad  = \frac{\Gamma(k-1)}{(4\pi)^{k-1}}\overline{L^{\reg}_{r}(\xi_{2-k}(g),k-1)} + \frac{(-i)^k \chi^{-2k}_{\theta}(\gamma)}{(2c)^{k-1}}\overline{b(0)},
\end{align}
where $b(0)$ is the constant term of the Fourier expansion of $\xi_{2-k}(g)$ at the cusp $i\infty$.
In particular, if $x = -\frac dc$, then we have
\begin{eqnarray*}
&&\lim_{t\to 0} \left(g^+(r+it) - \frac{\chi^{-2k}_{\theta}(\gamma)\overline{b(0)}}{2^{k-1}(k-1)}  \left( \frac{i}{ct}\right)^{2-k} \left(\frac{1}{c^2t} + \frac 1c\right)^{k-1}\right)\\
&& = \frac{\Gamma(k-1)}{(4\pi)^{k-1}}\overline{L^{\reg}_{r}(\xi_{2-k}(g),k-1)} + \frac{(-i)^k \chi^{-2k}_{\theta}(\gamma)}{(2c)^{k-1}}\overline{b(0)}.
\end{eqnarray*}
\end{thm}

\begin{rmk}
\begin{enumerate}
\item  When $\xi_{2-k}(g)$ is a cusp form, it was proved in \cite{CLR}.
\item Assume that $g$ is a harmonic weak Maass form of real weight $k$ on $\mathrm{SL}_2(\mathbb{Z})$ with $\eta$-multiplier system. Let $G^{-}$ be the regular part of $g^-$, i.e., $$G^{-}(z) =  \sum_{n\geq 1} c^-(n)B_k(2\pi ny)e^{2\pi inx}.$$ For $S=\sm 0 & -1 \\ 1 & 0 \esm$, the limit 
    $\lim_{z\to 0}(G^-|_{k}S-G^-)(z)$ was obtained in \cite{BDE}.
\end{enumerate}
\end{rmk}

\begin{proof} [\bf Proof of Theorem \ref{mocklimit}]
Note that
\begin{eqnarray*}
g^+(\gamma z) &=& \chi^{-2k}_{\theta}(\gamma) (cz+d)^{2-k}(g^+|_{2-k}\gamma)(z)= \chi^{-2k}_{\theta}(\gamma)(cz+d)^{2-k}\biggl[ (g^+|_{2-k}\gamma)(z) - g^+(z) + g^+(z) \biggr]\\
&=& \chi^{-2k}_{\theta}(\gamma)(cz+d)^{2-k}\biggl[ g^-(z) - (g^-|_{2-k}\gamma)(z) + g^+(z) \biggr]
\end{eqnarray*}
for $z\in\HH$.
Since $g^+(z)$ does not have neither principal part nor constant term in the Fourier expansion at the cusp $i\infty$, it decays exponentially as $z$ goes to $i\infty$.
Therefore,  we have
\[
\lim_{t\to 0} \chi^{-2k}_{\theta}(\gamma) \left( c\left( x+\frac{1}{c^2t}i \right) + d \right)^{2-k} g^+\left( x+ \frac{1}{c^2t}i \right) = 0.
\]
This implies that (\ref{mockquantum}) is equal to
\begin{eqnarray*}
&&\lim_{t\to 0} \biggl\{ \chi^{-2k}_{\theta}(\gamma) \left( c\left( x + \frac{1}{c^2t}i \right) + d \right)^{2-k} \biggl[ (g^- - g^-|_{2-k}\gamma)\left( x + \frac{1}{c^2t}i \right) \biggr]\\
&&-\frac{\chi^{-2k}_{\theta}(\gamma)\overline{b(0)}}{2^{k-1}(k-1)} \left( c\left(x + \frac{1}{c^2t} i\right)+d\right)^{2-k}\left(\left(\frac{1}{c^2t}+\frac{1}{c}\right)-i\left(x+\frac dc\right)\right)^{k-1}\biggr\}.
\end{eqnarray*}

Suppose that $z\in\HH$.
By Theorem \ref{maintheoremsection3} and Lemma \ref{transform}, we see that
\begin{eqnarray*}
\chi^{-2k}_{\theta}(\gamma) (cz+d)^{2-k}(g^- - g^-|_{2-k}\gamma)(z) = \frac{-1}{(-2i)^{k-1}} \overline{\int^{\reg}_{[i\infty, \gamma(i\infty)]} \xi_{2-k}(g)(\tau) (\tau-\overline{\gamma z})^{k-2}d\tau}.
\end{eqnarray*}
This implies that (\ref{mockquantum}) is equal to
\begin{eqnarray} \label{mockquantum2}
&&\lim_{t\to 0} \biggl[ \frac{-1}{(-2i)^{k-1}} \overline{\int^{\reg}_{[i\infty, \gamma(i\infty)]} \xi_{2-k}(g)(\tau)\left( \tau- \overline{\gamma\left( x+\frac{1}{c^2t}i\right)}\right)^{k-2}d\tau}\\
\nonumber &&-\frac{\chi^{-2k}_{\theta}(\gamma)\overline{b(0)}}{2^{k-1}(k-1)}\left( c\left(x+\frac{1}{c^2t}i\right)+d\right)^{2-k} \left(\left(\frac{1}{c^2t}+\frac{1}{c}\right)-i\left(x+\frac dc\right)\right)^{k-1}\biggr].
\end{eqnarray}
For any $z_0\in\HH$, by Lemma \ref{welldefinedforcusp},  we have
\begin{eqnarray*}
&&\frac{-1}{(-2i)^{k-1}} \overline{\int^{\reg}_{[i\infty, \gamma(i\infty)]} \xi_{2-k}(g)(\tau)\left(\tau-\overline{\gamma\left(x+\frac{1}{c^2t}i\right)}\right)^{k-2}d\tau}\\
&&= \frac{-1}{(-2i)^{k-1}} \biggl[\overline{\int^{\reg}_{[z_0, \gamma(i\infty)]} \xi_{2-k}(g)(\tau)\left(\tau-\overline{\gamma\left(x+\frac{1}{c^2t}i\right)}\right)^{k-2}d\tau} \\
&&- \overline{\int^{\reg}_{[z_0, i\infty]} \xi_{2-k}(g)(\tau)\left(\tau-\overline{\gamma\left(x+\frac{1}{c^2t}i\right)}\right)^{k-2}d\tau}\biggr]\\
&&=  \frac{-1}{(-2i)^{k-1}}\chi^{-2k}_{\theta}(\gamma) \left( c\left( x + \frac{1}{c^2t}i \right) + d\right)^{2-k} \overline{\int^{\reg}_{[\gamma^{-1}(z_0), i\infty]} \xi_{2-k}(g)(\tau)\left(\tau-\overline{x+\frac{1}{c^2t}i}\right)^{k-2}d\tau}\\
&& + \frac{1}{(-2i)^{k-1}}\overline{\int^{\reg}_{[z_0, i\infty]} \xi_{2-k}(g)(\tau)\left(\tau-\overline{\gamma\left(x+\frac{1}{c^2t}i\right)}\right)^{k-2}d\tau}.
\end{eqnarray*}
Suppose that $(cz+d)^{-k}\xi_{2-k}(g)(\gamma z)$ has a Fourier expansion of the form as in (\ref{fourierq}) for $\gamma = \sm a&b\\c&d\esm\in \SL_2(\ZZ)$.
We take $z_0 = \frac{a}{c} + \frac{1}{c}i \in \HH$.
Note that $\gamma^{-1}(z_0) = -\frac dc + \frac{1}{c}i$.
Then by Lemma \ref{generalpoint}, we obtain
\begin{eqnarray} \label{mockquantum3}
\lim_{t\to0}  && \biggl[ \frac{-1}{(-2i)^{k-1}}\chi^{-2k}_{\theta}(\gamma) \overline{\int^{\reg}_{[\gamma^{-1}(z_0), i\infty]} \xi_{2-k}(g)(\tau)\left(\tau-\overline{x+\frac{1}{c^2t}i}\right)^{k-2}d\tau}\\
\nonumber &&-\frac{\chi^{-2k}_{\theta}(\gamma)\overline{b(0)}}{2^{k-1}(k-1)} \left(\left(\frac{1}{c^2t}+\frac{1}{c}\right)-i\left(x+\frac dc\right)\right)^{k-1}\biggr] \left( c\left( x + \frac{1}{c^2t}i \right) + d\right)^{2-k} \\
\nonumber= \lim_{t\to 0} && \frac{-1}{(-2i)^{k-1}}\chi^{-2k}_{\theta}(\gamma) \left( c\left( x + \frac{1}{c^2t}i \right) + d\right)^{2-k} \left(\frac{-i}{2\pi}\right)^{k-1}\\
\nonumber &&\qquad \times \sum_{n\gg-\infty\atop n\neq0} \overline{\left(\frac{b(n)}{n^{k-1}}\right)}e^{2\pi in\frac{d}{c}}e^{2\pi n \frac{1}{c^2t}} \overline{A\left( 2\pi n \left(\frac{1}{c^2t} + \frac{1}{c}\right), 2\pi n\left( x + \frac{d}{c}\right), k-2\right)}.
\end{eqnarray}
By Lemma \ref{gammaasymptotic} and \ref{uniform}, the last limit in (\ref{mockquantum3}) is equal to
\begin{eqnarray*}
&&\lim_{t\to 0} \biggl[\frac{-1}{(4\pi)^{k-1}} \chi^{-2k}_{\theta}(\gamma) \left( c\left(x + \frac{1}{c^2t}i\right) + d\right)^{2-k}\\
&&\times \sum_{n\gg-\infty\atop n\neq 0} \overline{\left(\frac{b(n)}{n^{k-1}}\right)} e^{2\pi in\frac dc}e^{2\pi n\frac{1}{c^2t}}e^{-2\pi n\left(\frac{1}{c^2t} + \frac 1c\right)} \overline{\left(2\pi n\left(\frac{1}{c^2t} + \frac 1c\right)+2\pi ni\left(x+\frac dc\right)\right)^{k-2}}\biggr]\\
&&= \frac{(-i)^{k}}{2^{k-1}2\pi c^{k-2}} \sum_{n\gg-\infty\atop n\neq0} \frac{\overline{b_\gamma(n)}}{n} e^{2\pi in\frac dc}e^{-2\pi n\frac1c}.
\end{eqnarray*}
In a similar way, we also see that
\begin{align*}
&\lim_{t\to0} \frac{1}{(-2i)^{k-1}} \overline{\int^{\reg}_{[z_0,i\infty]} \xi_{2-k}(g)(\tau)\left(\tau-\overline{\gamma \left(x+\frac1{c^2t}i\right)}\right)^{k-2}d\tau}\\
&=\lim_{t\to0} \frac{1}{(-2i)^{k-1}} \biggl[ \left(\frac{-i}{2\pi}\right)^{k-1}\sum_{n\gg-\infty\atop n\neq0} \overline{\left(\frac{b(n)}{n^{k-1}}\right)} e^{2\pi in(-\frac ac)}e^{2\pi na(t)}\\
&\qquad \times\overline{A\left(2\pi n\left( \frac1c + a(t)\right), 2\pi n\left(c(t) - \frac ac\right), k-2\right)}\\
&\qquad -\frac{\overline{b(0)}(-i)^{k-1}}{k-1}\overline{\left(\left(\frac1c +a(t)\right) + i\left(c(t)-\frac ac\right)\right)^{k-1}}\biggr]\\
& =\frac{1}{(-2i)^{k-1}} \biggl[ \left(\frac{-i}{2\pi}\right)^{k-1} \sum_{n\gg-\infty\atop n\neq 0} \overline{\left(\frac{b(n)}{n^{k-1}}\right)} e^{2\pi in(-\frac ac)}A\left( 2\pi n\frac1c, 0, k-2\right) - \frac{\overline{b(0)}(-i)^{k-1}}{c^{k-1}(k-1)}\biggr]\\
&=\frac{1}{(4\pi)^{k-1}} \sum_{n\gg-\infty\atop n\neq0} \overline{\left(\frac{b(n)}{n^{k-1}}\right)} e^{2\pi in(-\frac ac)}\Gamma\left(k-1, \frac{2\pi n}{c}\right) - \frac{\overline{b(0)}}{(2c)^{k-1}(k-1)},
\end{align*}
where $a(t) = \mathrm{Im}\left( \gamma\left(x+\frac{1}{c^2t}i\right)\right)$ and $c(t) = \mathrm{Re}\left( \gamma\left(x+\frac{1}{c^2t}i\right)\right)$ for $t>0$.
Since $L^{\reg}_r(\xi_{2-k}(g), s)$ can be written as in (\ref{tindependent}), we see that
(\ref{mockquantum}) is equal to
\begin{eqnarray*}
&&\frac{1}{(4\pi)^{k-1}} \sum_{n\gg-\infty\atop n\neq0} \overline{\left(\frac{b(n)}{n^{k-1}}\right)}e^{2\pi in(-\frac ac)} \Gamma\left(k-1, \frac{2\pi n}{c}\right)\\
&&+\frac{(-i)^k}{2^{k-1}2\pi c^{k-2}}\sum_{n\gg-\infty\atop n\neq0} \frac{\overline{b_\gamma(n)}}{n}e^{2\pi in\frac dc}\Gamma\left(1, \frac{2\pi n}{c}\right) - \frac{\overline{b(0)}}{(2c)^{k-1}(k-1)}\\
&&=\frac{\Gamma(k-1)}{(4\pi)^{k-1}}\overline{L^{\reg}_{r}(\xi_{2-k}(g), k-1)} + \frac{(-i)^k \chi^{-2k}_{\theta}(\gamma)\overline{b(0)}}{(2c)^{k-1}}.
\end{eqnarray*}
\end{proof}

Let $g\in H_{2-k}(\Gamma_0(4N))$ and  $r$ be a rational number which is equivalent to $i\infty$ under the action of $\Gamma_0(4N)$.
Then we define $Q_g(r)$ by
\[
\lim_{t\to 0} \left[g^+\left(\gamma\left(x+\frac{1}{c^2t}i\right)\right) -\frac{\chi^{-2k}_{\theta}(\gamma)\overline{b(0)}}{2^{k-1}(k-1)} \left( c\left(x+\frac{1}{c^2t}i\right)+d\right)^{2-k}\left(\left(\frac{1}{c^2t}+\frac{1}{c}\right)-i\left(x+\frac dc\right)\right)^{k-1}\right].
\]
The following corollary shows that the period function of $Q_g$ is the limit value of the period function of $g^+$ if the constant term of $\xi_{2-k}(g)$ at the cusp $i\infty$ vanishes.

\begin{cor} \label{limitperiod}
Let $g\in H_{2-k}(\Gamma_0(4N))$ such that
$g^+$ does not have neither principal part nor constant term in the Fourier expansion at the cusp $i\infty$.
Suppose that
the constant term of $\xi_{2-k}(g)$ at the cusp $i\infty$ vanishes.
Let $\psi(g)_\gamma = g^+ - g^+|_{2-k}\gamma$ for $\gamma\in \Gamma_0(4N)$.
Then, we have
\[
(Q_g - Q_g|_{2-k}\gamma')(r) = \lim_{t\to\infty} \psi(g)_{\gamma'}(\gamma (x+it))
\]
for $\gamma'\in \Gamma_0(4N)$ and $r\in\QQ$ equivalent to $i\infty$ under the action of $\Gamma_0(4N)$ with $r\neq (\gamma')^{-1}(i\infty)$.
\end{cor}

\begin{proof} [\bf Proof of Corollary \ref{limitperiod}]
Suppose that $(cz+d)^{-k}\xi_{2-k}(g)(\gamma z)$ has a Fourier expansion of the form as in (\ref{fourierq}) for $\gamma = \sm a&b\\c&d\esm\in \SL_2(\ZZ)$.
By the assumption, we have $b(0)= 0$.
By (\ref{mockquantum2}), we have
\begin{eqnarray} \label{limitperiod1}
&&Q_g(r) = \lim_{t\to 0} \biggl[ \frac{-1}{(-2i)^{k-1}} \overline{\int^{\reg}_{[i\infty, \gamma(i\infty)]} \xi_{2-k}(g)(\tau)\left( \tau- \overline{\gamma\left( x+\frac{1}{c^2t}i\right)}\right)^{k-2}d\tau}\biggr]\\
\nonumber &&= \lim_{t\to \infty} \biggl[ \frac{-1}{(-2i)^{k-1}} \overline{\int^{\reg}_{[i\infty, \gamma(i\infty)]} \xi_{2-k}(g)(\tau)\left( \tau- \overline{\gamma\left( x+it\right)}\right)^{k-2}d\tau}\biggr],
\end{eqnarray}
where $\gamma = \sm a&b\\c&d\esm \in \Gamma_0(4N)$ with $\gamma(i\infty) = r$.
Note that (\ref{limitperiod1}) holds for any $x\in\RR$, and we have
\begin{eqnarray} \label{limitperiod2}
&&(Q_g|_{2-k}\gamma')(r) = \chi^{2k}_{\theta}(\gamma')(c'r+d')^{k-2}Q_g(\gamma' r)\\
\nonumber && = \chi^{2k}_{\theta}(\gamma')(c'r+d')^{k-2}\lim_{t\to \infty} \biggl[ \frac{-1}{(-2i)^{k-1}} \overline{\int^{\reg}_{[i\infty, \gamma'\gamma(i\infty)]} \xi_{2-k}(g)(\tau)\left( \tau- \overline{
\gamma'\gamma\left( x+it\right)}\right)^{k-2}d\tau}\biggr]\\
\nonumber &&= \lim_{t\to \infty} \biggl[ \frac{-\chi^{2k}_{\theta}(\gamma') j\left(\gamma', \gamma\left(x+it\right)\right)^{k-2}}{(-2i)^{k-1}} \overline{\int^{\reg}_{[i\infty, \gamma'\gamma(i\infty)]} \xi_{2-k}(g)(\tau)\left( \tau- \overline{
\gamma'\gamma\left( x+it\right)}\right)^{k-2}d\tau}\biggr],
\end{eqnarray}
where $\gamma' = \sm a'&b'\\ c'&d'\esm$.
By Lemma \ref{transform}, (\ref{limitperiod2}) is equal to
\begin{eqnarray} \label{limitperiod3}
&&\lim_{t\to \infty} \biggl[ \frac{-\chi^{2k}_{\theta}(\gamma') j\left(\gamma', \gamma\left(x+it\right)\right)^{k-2}\chi^{-2k}_{\theta}(\gamma')j\left(\gamma'^{-1}, \gamma'\gamma\left(x+it\right)\right)^{k-2}}{(-2i)^{k-1}}\\
\nonumber &&\times  \overline{\int^{\reg}_{[\gamma'^{-1}(i\infty), \gamma(i\infty)]} \xi_{2-k}(g)(\tau)\left( \tau- \overline{
\gamma\left( x+it\right)}\right)^{k-2}d\tau}\biggr]\\
\nonumber && = \lim_{t\to\infty} \frac{-1}{(-2i)^{k-1}}\biggl[\overline{\int^{\reg}_{[\gamma'^{-1}(i\infty), \gamma(i\infty)]} \xi_{2-k}(g)(\tau)\left( \tau- \overline{
\gamma\left( x+it\right)}\right)^{k-2}d\tau}\biggr].
\end{eqnarray}
Here, we used the fact that $j(\gamma'^{-1}, \gamma' z) = j(\gamma', z)^{-1}$ for any $z\in \HH$.
If we combine (\ref{limitperiod1}) and (\ref{limitperiod3}), then we have
\begin{equation} \label{limitperiod4}
Q_g(r) - (Q_g|_{2-k}\gamma')(r) = \lim_{t\to\infty} \frac{-1}{(-2i)^{k-1}}\biggl[\overline{\int^{\reg}_{[i\infty, \gamma'^{-1}(i\infty)]} \xi_{2-k}(g)(\tau)\left( \tau- \overline{
\gamma\left( x+it\right)}\right)^{k-2}d\tau}\biggr].
\end{equation}
By Lemma \ref{periodfunction}, (\ref{limitperiod4}) is equal to
\begin{equation} \label{limitperiod5}
-\lim_{t\to\infty} (g^- - g^-|_{2-k}\gamma')(\gamma (x+it)) = \lim_{t\to\infty} (g^+ - g^+|_{2-k}\gamma')(\gamma (x+it))= \lim_{t\to\infty} \psi(g)_{\gamma'}(\gamma (x+it)).
\end{equation}
\end{proof}

\begin{rmk}
Let $g\in H_{2-k}(\Gamma_0(4N))$ such that the constant term of $\xi_{2-k}(g)$ vanishes.
The function $Q_{g}(r)$ for a rational number $r = \frac ab$ with $b\equiv 0\ (\m\ 4N)$ and $\mathrm{gcd}(a,b) = 1$
 induces a function on $\QQ$  in the form of $\frac{\Gamma(k-1)}{(4\pi)^{k-1}}\overline{L^{\reg}_{r}(\xi_{2-k}(g),k-1)}$ by (\ref{mockquantum}).
The equality (\ref{limitperiod4}) shows that this function can be understood as a quantum modular form, which is a new modular object introduced by Zagier \cite{Zag2}.
Quantum modular forms have been   studied in connection with numerous different topics, such as Maass forms, mock theta functions, and strongly unimodal sequences (for example, see \cite{Bru0, CLR, FOR}).
\end{rmk}

In the following corollary, we consider the case $r = (\gamma')^{-1}(i\infty)$.

\begin{cor} \label{psilimit}
Let $g\in H_{2-k}(\Gamma_0(4N))$.
Suppose that $g^+$ does not have neither principal part nor constant term in the Fourier expansion at the cusp $i\infty$.
Let $\gamma = \sm a&b\\c&d\esm\in \Gamma_0(4N)$ with $r = \gamma(i\infty) \in \QQ$ and $c>0$.
Then, we have
\begin{eqnarray} \label{psilimit2}
&&\lim_{t\to 0}\left[ \psi(g)_{\gamma^{-1}}(r+it) - \frac{\chi^{-2k}_{\theta}(\gamma)\overline{b(0)}}{2^{k-1}(k-1)}  \left( \frac{i}{ct}\right)^{2-k} \left(\frac{1}{c^2t} + \frac 1c\right)^{k-1}\right]\\
\nonumber && = \frac{\Gamma(k-1)}{(4\pi)^{k-1}}\overline{L^{\reg}_{r}(\xi_{2-k}(g),k-1)} + \frac{(-i)^k \chi^{-2k}_{\theta}(\gamma)}{(2c)^{k-1}}\overline{b(0)},
\end{eqnarray}
where $b(0)$ is the constant term of the Fourier expansion of $\xi_{2-k}(g)$ at the cusp $i\infty$.
\end{cor}

\begin{proof} [\bf Proof of Corollary \ref{psilimit}]
Note that
\[
\psi(g)_{\gamma^{-1}}(r+it) = g^+(r+it) - (g^+|_{2-k}\gamma^{-1})(r+it)
\]
for $t>0$, and
\[
\lim_{t\to 0} (g^+|_{2-k}\gamma^{-1})(r+it) = 0
\]
since $g^+$ does not have neither principal part nor constant term in the Fourier expansion at the cusp $i\infty$. Therefore, (\ref{psilimit2}) is equal to
\[
\lim_{t\to 0}\left[ g^+(r+it) - \frac{\chi^{-2k}_{\theta}(\gamma)\overline{b(0)}}{2^{k-1}(k-1)}  \left( \frac{i}{ct}\right)^{2-k} \left(\frac{1}{c^2t} + \frac 1c\right)^{k-1}\right].
\]
By Theorem \ref{mocklimit}, we get the desired result.
\end{proof}

\section{Proofs of Theorem \ref{trace} and Corollary \ref{period}} \label{proof1}
In this section, we prove Theorem \ref{trace} and Corollary \ref{period}. Theorem \ref{trace} is implied by Theorem Theorem \ref{mocklimit}.

\begin{proof} [\bf Proof of Theorem \ref{trace}]
Note that $H(f)$ is the holomorphic part of a harmonic weak Maass form $\hat{H}(f) \in H_{\frac12}(\Gamma_0(4N))$ and $\xi_{\frac12}(\hat{H}(f)) = -2W(f)$ by Corollary \ref{scalar}.
The constant term of $\xi_{\frac12}(\hat{H}(f))$ is $-\overline{\mathrm{Tr}_0(f)}$.
Therefore, by Theorem \ref{mocklimit}, we have
\begin{eqnarray} \label{tracelimit}
&&\lim_{t\to0} \left( \sum_{d>0}\mathrm{Tr}_d(f)e^{2\pi id(r+it)} - \sqrt{2}\chi^{-3}_{\theta}(\gamma)(-\mathrm{Tr}_0(f))\left(\frac{i}{ct}\right)^{\frac12}\left(\frac{1}{c^2t}+\frac 1c\right)^{\frac12}\right)\\
\nonumber && = -\overline{L^{\reg}_{r}\left(W(f),\frac12\right)} + \frac{(-i)^{\frac32}\chi^{-3}_{\theta}(\gamma)(-\mathrm{Tr}_0(f))}{\sqrt{2c}},
\end{eqnarray}
where $\gamma = \sm a&b\\c&d\esm \in \Gamma_0(4N)$.
By the definition of the multiplier system $\chi_{\theta}$, we have
\[
\chi_{\theta}(\gamma) = \varepsilon_d^{-1} \left(\frac{c}{d}\right).
\]
Since $a \equiv d\ (\mathrm{mod}\ 4)$, we see that $\varepsilon_d = \varepsilon_a$.
Let $c = 2^n c'$ such that $n\in\NN$ and $c' \equiv 1\ (\mathrm{mod}\ 2)$.
Since $\sm a&b\\c&d\esm \in \Gamma_0(4N)$, we have $n\geq2$ and $\mathrm{gcd}(c', d) = 1$.
By the definition of the symbol $\left(\frac cd\right)$, we obtain
\begin{eqnarray*}
\biggl( \frac cd\biggr) = \left( \frac 2d\right)^n \left( \frac{c'}d\right)
= \left( \frac 2d\right)^n \left( \frac d{c'}\right) (-1)^{\frac{c'-1}2\frac{d-1}2}.
\end{eqnarray*}
Note that
\[
\left( \frac d{c'}\right) \biggl(\frac a{c'} \biggr) = \left( \frac{ad}{c'}\right) = \left( \frac{ad-bc}{c'}\right) = \left( \frac 1{c'}\right) = 1.
\]
This implies that $\left( \frac d{c'}\right) = \left( \frac a{c'}\right)$.
Since $a\equiv d\ (\mathrm{mod}\ 4)$, we have $(-1)^{\frac{c'-1}2 \frac{d-1}2} = (-1)^{\frac{c'-1}2 \frac{a-1}2}$.
Therefore, we see that
\[
\biggl( \frac{c'}{d}\biggr) = \left( \frac d{c'}\right) (-1)^{\frac{c'-1}2\frac{d-1}2} = \biggl( \frac a{c'}\biggr) (-1)^{\frac{c'-1}2\frac{a-1}2} = \left( \frac{c'}a\right).
\]
If $n = 2$, then we obtain
\[
\left( \frac 2d\right)^n = 1 = \left( \frac 2a\right)^n.
\]
If $n\geq3$, then we have $ad \equiv 1\ (\mathrm{mod}\ 8)$, and hence $a \equiv d\ (\mathrm{mod}\ 8)$. Therefore, we see that
\[
\left( \frac 2d\right) = (-1)^{\frac{d^2-1}8} = (-1)^{\frac{a^2-1}8}=\left( \frac 2a\right).
\]
If we combine all these equalities, then we obtain $\left( \frac cd\right) = \left( \frac ca\right)$.
This implies that
\[
\chi_{\theta}(\gamma) = \varepsilon_a^{-1} \left(\frac ca\right).
\]
With this, we can see that (\ref{tracelimit}) can be written as
\begin{equation*}
\lim_{t \to 0} \left( \sum_{d>0} \Tr_d(f)e^{2\pi id(r+it)}+ \frac{2c_r\sqrt{1+ct}}{ct} \right) = - \overline{L^{\reg}_r\left(W(f), \frac12\right)}+c_r,
\end{equation*}
where $c_r=\frac{e^{i\pi/4}\varepsilon_a^{3}}{\sqrt{2c}}\left(\frac{c}{a}\right)\Tr_0(f)$.
\end{proof}

\begin{proof} [\bf Proof of Corollary \ref{period}]
By the same way as in the proof of Theorem \ref{trace}, Corollary \ref{psilimit} implies that
\[
\lim_{t \to 0} \left( \psi(\hat{H}(f))_{\gamma^{-1}}(r+it) + \frac{2c_r\sqrt{1+ct}}{ct} \right) = - \overline{L^{\reg}_r\left(W(f), \frac12\right)}+c_r.
\]
This is the desired result.
\end{proof}

\section{Eichler-Shimura cohomology for weakly holomorphic modular forms} \label{reviewcohomology}

Knopp \cite{Kno} used holomorphic functions on the complex upper half plane $\HH$ to study Eichler-Shimura cohomology for real weights.
In \cite{BCD}, Bruggeman, Choie, and Diamantis worked with functions on the complex lower half plane $\HH^-$ instead of $\HH$ to establish the Eichler-Shimura cohomology theory for holomorphic automorphic forms.
This approach avoided complex conjugation when dealing with period functions and led to a $\CC$-linear map from modular forms to cohomology classes.
As remarked in \cite{BCD}, these two methods in \cite{Kno} and \cite{BCD} are equivalent by the involution $\iota$ in (1.6) and (1.12) in \cite{BCD}.
In this section, we review the Eichler-Shimura cohomology theory for weakly holomorphic modular forms by following Bruggeman, Choie, and Diamantis \cite{BCD}. And then, by using a  result in \cite{BCD}, we prove that there is no nonzero harmonic weak Maasss form of half integral weight such that its holomorphic part is zero.

Now, we introduce some notions in \cite{BCD}.
 A set $\Omega\subset \mathbb{P}^1_{\CC}$ is an \emph{excised neighborhood} of $\HH^-\cup \mathbb{P}^1_{\RR}$, if it contains a set of the form
\begin{equation*}
U \setminus \bigcup_{\xi\in E} W_{\xi},
\end{equation*}
where $U$ is a standard neighborhood of $\HH^-\cup \mathbb{P}^1_{\RR}$ in $\mathbb{P}^1_{\CC}$.
Here, $E$ is a finite subset of $\mathbb{P}^1_{\RR}$, called the \emph{excised set}, and $W_{\xi}$ has the form
\[
W_{\xi} = \{ h_{\xi}z \in \HH\ |\ |\mathrm{Re}(z)| \leq a\ \text{and}\ \mathrm{Im}(z)>\epsilon\},
\]
with $h_{\xi}\in \SL_2(\RR)$ such that $h_{\xi}\infty = \xi$, and $a, \epsilon>0$.
Instead of excised neighborhood of $\HH^-\cup \mathbb{P}^1_{\RR}$ with excised set E, we shall write \emph{$E$-excised neighborhood}.

For $k\in\frac12\ZZ$ and a function $f$, we define the function $\mathrm{prj}_{2-k}f$ by
\[
(\mathrm{prj}_{2-k}f)(z) := (i-z)^{2-k}f(z).
\]
For $\xi_1, \cdots, \xi_n\in \mathbb{P}^1_{\RR}$, we define spaces of \emph{excised semi-analytic vector}
\[
\mathcal{D}^{\omega, \mathrm{exc}}_{2-k}[\xi_1, \cdots, \xi_n] := \mathrm{prj}^{-1}_{2-k}\lim_{\longrightarrow} \mathcal{O}(\Omega),
\]
where $\Omega$ runs over the $\{\xi_1, \cdots, \xi_n\}$-excised neighborhoods and $\mathcal{O}$ denotes the sheaf of holomorphic functions on $\mathbb{P}^1_{\CC}$.
We define
\[
\mathcal{D}^{\omega^0, \mathrm{exc}}_{2-k} := \lim_{\longrightarrow} \mathcal{D}^{\omega, \mathrm{exc}}_{2-k}[a_1, \cdots, a_n],
\]
where $\{a_1, \cdots, a_n\}$ runs over the finite sets of cusps of $\Gamma_0(4N)$.
Here, for $f\in \mathcal{D}^{\omega^0, \mathrm{exc}}_{2-k},\ z\in \HH^-$ and $\gamma \in \Gamma_0(4N)$, the group $\Gamma_0(4N)$ acts on $\mathcal{D}^{\omega^0, \mathrm{exc}}_{2-k}$ by
\[
(f|^-_{2-k}\gamma)(z) := \chi^{-2k}_{\theta}(\gamma) (cz+d)^{k-2} f(\gamma z).
\]

We recall the basic definitions of group cohomology (see Section 1.4 in \cite{BCD}).
Let $V$ be a right $\CC[\Gamma]$-module.
A {\it $1$-cocycle} is a map $\psi: \Gamma_0(4N) \to V$ such that
\[
\psi(\gamma \delta) = \psi(\gamma)\cdot \delta + \psi(\delta)
\]
for all $\gamma, \delta\in \Gamma_0(4N)$.
A {\it $1$-coboundary} is a map $\psi: \Gamma_0(4N) \to V$ such that
\[
\psi(\gamma) = a\cdot \gamma - a
\]
for some $a\in V$ not depending on $\gamma$.

Let $Z^1(\Gamma_0(4N); V)$ be the space of $1$-cocycles, and $B^1(\Gamma_0(4N); V) \subset Z^1(\Gamma_0(4N); V)$ the space of $1$-coboundaries.
The \emph{parabolic cohomology group} $H^1_{\mathrm{pb}}(\Gamma_0(4N); V)$ is
\[
H^1_{\mathrm{pb}}(\Gamma_0(4N); V) := Z^1_{\mathrm{pb}}(\Gamma_0(4N);V)/ B^1(\Gamma_0(4N); V),
\]
where
\begin{eqnarray*}
Z^1_{\mathrm{pb}}(\Gamma_0(4N); V) &:=& \biggl\{ \psi\in Z^1(\Gamma_0(4N); V)\ \bigg|\ \psi(\pi) \in V \cdot(\pi-1)\ \text{for all parabolic $\pi\in \Gamma_0(4N)$}\biggr\}.
\end{eqnarray*}

For a fixed $z_0\in \HH$ and $f\in M^!_{k}(\Gamma_0(4N))$, consider the map
$\phi^{z_0}_f : \Gamma_0(4N) \to \mathcal{D}^{\omega^0, \mathrm{exc}}_{2-k}$ defined by
\[
\phi^{z_0}_f(\gamma) := E(f)|^-_{2-k}\gamma - E(f),
\]
where $E(f)(z) := \int^{\bar{z}}_{z_0} f(\tau)(\tau-z)^{k-2}d\tau$.
Note that
\[
\phi^{z_0}_f(\gamma)(z) =\int^{z_0}_{\gamma^{-1}z_0} f(\tau)(\tau-z)^{k-2}d\tau
\]
(see Lemma 5.1 in \cite{BCD}).
This gives the following injective linear map.

\begin{thm} \cite[Theorem E]{BCD} \label{BCDinjective}
If $k$ is not an integer, then the assignment $\phi^{z_0}_f$ is a $1$-cocycle, and $f \mapsto \phi^{z_0}_f$ induces an injective linear map
\[
r_k : M^!_{k}(\Gamma_0(4N)) \to H^1_{\mathrm{pb}}(\Gamma_0(4N); \mathcal{D}^{\omega^0, \mathrm{exc}}_{2-k}).
\]
\end{thm}

The following lemma shows that the regularized integral $\int^{\mathrm{reg}}_{[z_0, i\infty]} f(\tau)(\tau-z)^{k-2}d\tau$  is an element of $\mathcal{D}^{\omega^0, \mathrm{exc}}_{2-k}$.

\begin{lem} \label{coboundary}
The function $\int^{\reg}_{[z_0,i\infty]} f(\tau)(\tau-z)^{k-2}d\tau$ is an element of $\mathcal{D}^{\omega^0, \mathrm{exc}}_{2-k}$.
\end{lem}

\begin{proof} [\bf Proof of Lemma \ref{coboundary}]
Note that the regularized integral
\begin{equation} \label{integralconnection}
\int^{\reg}_{[z_0,i\infty]} f(\tau)(\tau-z)^{k-2}d\tau
\end{equation}
is well defined for $z\in\CC$  if $\mathrm{Im}(z_0)> \mathrm{Im}(z)$ by Lemma \ref{generalpoint}.
Now, we extend this integral to $\CC$ as a function of $z$ as follows.

We fix $z_1\in\HH$ such that $\mathrm{Im}(z_1)\geq\mathrm{Im}(z_0)$.
Let $\gamma$ be a path in $\HH$ joining $z_0$ and $z_1$.
We consider the following function $J_\gamma(z_0, z_1)(z)$ defined by
\begin{equation} \label{extension}
J_\gamma(z_0, z_1)(z) := \int_{z_0}^{z_1} f(\tau)(\tau-z)^{k-2}d\tau
\end{equation}
for $z\in \CC$ such that $\mathrm{Im}(z) < \mathrm{Im}(z_1)$, where we take the path $\gamma$.
Then the integral in (\ref{extension}) is well defined if the path $\gamma$ does not pass through the line $\{z-t\ |\ t\geq0\}$ since we use the convention that $z = |z|e^{i\mathrm{arg}(z)},\ -\pi< \mathrm{arg}(z) \leq \pi$.
Let $D(z_0, z_1, \gamma)$ be the set of complex numbers $z$ such that the integral in (\ref{extension}) is well defined at $z$.
Then, $J_{\gamma}(z_0, z_1)(z)$ is a holomorphic function on $D(z_0, z_1, \gamma)$ as a function of $z$.
The following sum
\begin{equation} \label{extension2}
H_{z_1, \gamma}(z) := J_\gamma(z_0, z_1)(z) + \int^{\mathrm{reg}}_{[z_1, i\infty]} f(\tau)(\tau-z)^{k-2}d\tau,
\end{equation}
is well defined and induces a function of $z$ if  $\mathrm{Im}(z) < \mathrm{Im}(z_1)$ and $z\in D(z_0, z_1, \gamma)$.
This function is equal to the function in (\ref{integralconnection}) as a function of $z$ if $\mathrm{Im}(z_0) > \mathrm{Im}(z)$.
If $H_{z_1, \gamma}(z)$ and $H_{z_1', \gamma'}(z)$ are well defined, then we have $H_{z_1, \gamma}(z) = H_{z_1', \gamma'}(z)$ by the same argument as in the proof of Lemma \ref{independent}.
For a given $z\in\CC$, we can find $z_1\in \HH$ and $\gamma$ such that $\mathrm{Im}(z) < \mathrm{Im}(z_1)$ and $z\in D(z_0, z_1, \gamma)$.
Therefore, we have an extension of the function in (\ref{integralconnection}) to $\CC$ as a function of $z$.
Now, it is enough to show that $\int^{\mathrm{reg}}_{[z_1, i\infty]} f(\tau)(\tau-z)^{k-2}d\tau$ is holomorphic for $z$ with $\mathrm{Im}(z)<\mathrm{Im}(z_1)$
since the function $J_\gamma(z_0, z_1)(z)$  is holomorphic as a function of $z$ on $D(z_0, z_1,\gamma)$.

Suppose that $\mathrm{Im}(z)<\mathrm{Im}(z_0)$.
We will show that the function in (\ref{integralconnection}) is holomorphic at $z$.
If $f$ has a Fourier expansion  as in (\ref{Fourierinfinite}), then, by Lemma \ref{generalpoint}, (\ref{integralconnection}) is equal to
\begin{eqnarray*}
&& \left(\frac{i}{2\pi}\right)^{k-1} \sum_{n\gg-\infty\atop n\neq0} \frac{a(n)}{n^{k-1}} e^{2\pi inx_0}e^{-2\pi ny} A(2\pi n(y_0-y), 2\pi n(x-x_0), k-2)\\
&&-a(0)\frac{i^{k-1}(y_0-y + i(x-x_0))^{k-1}}{k-1},
\end{eqnarray*}
where $z_0 = x_0 + iy_0$ and $z = x + iy$.
Since $-y + ix = iz$, we see that
\[
-a(0)\frac{i^{k-1}(y_0-y + i(x-x_0))^{k-1}}{k-1} = -a(0)\frac{i^{k-1}(y_0-ix_0 +iz)^{k-1}}{k-1}
\]
is holomorphic as a function of $z$ if $\mathrm{Im}(z)<\mathrm{Im}(z_0)$.

Note that
we have
\begin{eqnarray*} \label{partial1}
&&\frac{\partial}{\partial x} \left( \frac{a(n)}{n^{k-1}} e^{2\pi inx_0}e^{-2\pi ny} A(2\pi n(y_0-y), 2\pi n(x-x_0), k-2) \right)\\
&&= \frac{a(n)}{n^{k-1}} e^{2\pi inx_0}e^{-2\pi ny} 2\pi n(k-2)iA(2\pi n(y_0-y), 2\pi n(x-x_0), k-3).
\end{eqnarray*}
Since the series
\[
\sum_{n=1}^\infty \frac{a(n)}{n^{k-1}} e^{2\pi inx_0}e^{-2\pi ny} 2\pi n(k-2)iA(2\pi n(y_0-y), 2\pi n(x-x_0), k-3)
\]
converges uniformly  by Lemma \ref{uniform}, we obtain
\begin{eqnarray*}
&&\frac{\partial}{\partial x}\left( \left(\frac{i}{2\pi}\right)^{k-1} \sum_{n\gg-\infty\atop n\neq0} \frac{a(n)}{n^{k-1}} e^{2\pi inx_0}e^{-2\pi ny} A(2\pi n(y_0-y), 2\pi n(x-x_0), k-2) \right) \\
&& = \left(\frac{i}{2\pi}\right)^{k-1} \sum_{n\gg-\infty\atop n\neq 0} \frac{a(n)}{n^{k-1}} e^{2\pi inx_0}e^{-2\pi ny} 2\pi n(k-2)iA(2\pi n(y_0-y), 2\pi n(x-x_0), k-3).
\end{eqnarray*}
By the same way, we see that
\begin{eqnarray*}
&&\frac{\partial}{\partial y}\left( \left(\frac{i}{2\pi}\right)^{k-1} \sum_{n\gg-\infty\atop n\neq0} \frac{a(n)}{n^{k-1}} e^{2\pi inx_0}e^{-2\pi ny} A(2\pi n(y_0-y), 2\pi n(x-x_0), k-2) \right) \\
&&= \left(\frac{i}{2\pi}\right)^{k-1} \sum_{n\gg-\infty\atop n\neq 0}
\frac{a(n)}{n^{k-1}} e^{2\pi inx_0}\biggl[ -2\pi ne^{-2\pi ny} A(2\pi n(y_0-y), 2\pi n(x-x_0), k-2)\\
&& +2\pi n e^{-2\pi ny_0}(2\pi n(y_0-y)+i2\pi n(x-x_0))^{k-2}\biggr].
\end{eqnarray*}
If we use the integration by parts, then we have
\begin{eqnarray*}
&&A(2\pi n(y_0-y), 2\pi n(x-x_0), k-2)\\
 &&= e^{-2\pi n(y_0-y)}(2\pi n(y_0-y)+i2\pi n(x-x_0))^{k-2}
 + (k-2)A(2\pi n(y_0-y), 2\pi n(x-x_0), k-3).
\end{eqnarray*}
From this, we obtain
\begin{eqnarray*}
\frac12\left(\frac{\partial}{\partial x} + i\frac{\partial}{\partial y}\right) \left(\left(\frac{i}{2\pi}\right)^{k-1} \sum_{n\gg-\infty\atop n\neq0} \frac{a(n)}{n^{k-1}} e^{2\pi inx_0}e^{-2\pi ny} A(2\pi n(y_0-y), 2\pi n(x-x_0), k-2) \right) = 0.
\end{eqnarray*}
Since $\frac{\partial}{\partial\bar{z}} = \frac12\left(\frac{\partial}{\partial x} + i\frac{\partial}{\partial y}\right)$, this completes the proof.
\end{proof}

\begin{rmk}
One can prove Lemma \ref{coboundary} by using the approach that Bruggeman, Choie, and Diamantis followed in \cite{BCD} with explicit manipulations of contour integrals.
\end{rmk}

Theorem C of \cite{BCD} provides a method to study harmonic weak Maass forms.
The following theorem is an example of the application of Theorem E in \cite{BCD}.

\begin{thm} \label{nonexistence}
Suppose that $k\in \frac12 + \ZZ$.
Then, there is no  nonzero harmonic weak Maass form $g\in H_{2-k}(\Gamma_0(4N))$ such that $g^+ = 0$.
\end{thm}

\begin{proof} [\bf Proof of Theorem \ref{nonexistence}]
Suppose that  $g\in H_{2-k}(\Gamma_0(4N))$ is a harmonic weak Maass form such that $g^+ = 0$.
Note that by (\ref{nonholomorphicpart}) we have
\[
\int^{\reg}_{[\bar{z},i\infty]} \xi_{2-k}(g)(\tau)(\tau-z)^{k-2}d\tau
 = -(2i)^{k-1}\overline{g^-(\overline{z})}
\]
for $z\in\HH^-$.
Since $g^+ = 0$, we obtain
\[
(g^-|_{2-k}\gamma - g^-)(z)  = 0
\]
for any $\gamma \in \Gamma_0(4N)$ and $z\in\HH$.
Hence, we have
\begin{equation} \label{nonexistence2}
(\overline{g^-}|^-_{2-k}\gamma - \overline{g^-})(\overline{z}) = 0
\end{equation}
for $z\in\HH^-$.

Note that
\begin{equation} \label{integralrelation}
\int^{\reg}_{[\bar{z},i\infty]} \xi_{2-k}(g)(\tau)(\tau-z)^{k-2}d\tau = \int^{\reg}_{[z_0,i\infty]} \xi_{2-k}(g)(\tau)(\tau-z)^{k-2}d\tau + \int^{z_0}_{\bar{z}} \xi_{2-k}(g)(\tau)(\tau-z)^{k-2}d\tau
\end{equation}
for $z\in \HH^-$.
We denote by $I_1(z)$ (resp. $I_2(z)$) the integral in the left hand side of (\ref{integralrelation}) (resp. the first  integral in the right hand side of (\ref{integralrelation})).
From (\ref{nonexistence2}) and (\ref{integralrelation}), we obtain
\begin{eqnarray*}
\phi^{z_0}_{\xi_{2-k}(g)}(\gamma) &=&  \left[I_2|^-_{2-k}\gamma - I_2\right] -  \left[I_1|^-_{2-k}\gamma - I_1\right]= \left[I_2|^-_{2-k}\gamma - I_2\right] .
\end{eqnarray*}
By Lemma \ref{coboundary}, $I_2$ is an element of $\mathcal{D}^{\omega^0, \mathrm{exc}}_{2-k}$, and hence $\phi^{z_0}_{\xi_{2-k}(g)}$ is a $1$-coboundary.
By Theorem \ref{BCDinjective}, we have $r_{k}(\xi_{2-k}(g)) = 0$, which implies that $\xi_{2-k}(g) = 0$ since $r_{k}$ is injective.
Therefore, we have $g=0$ since $\xi_{2-k}$ is also injective.
\end{proof}

\section{Proof of Theorem \ref{determine}} \label{proofs}

In this section, 
we prove Theorem \ref{determine}. The following lemma is needed to prove Theorem \ref{determine},.

\begin{lem} \label{Heckeoperator}
Let $g \in H_{2-k}(\Gamma_0(4N))$. Let $h(z) = \frac{1}{M^2}g(M^2z)$ for $M\in\NN$.
Then, we have
\[
Q_h(r) = \frac{1}{M^{2k-2}} Q_{g}(M^2r)
\]
for a rational number $r$ equivalent to $i\infty$ under the action of $\Gamma_0(4NM^2)$.
\end{lem}

\begin{proof} [\bf Proof of Lemma \ref{Heckeoperator}]
Note that by (\ref{mockquantum2}), $Q_{h}(r)$ is equal to
\begin{eqnarray} \label{Hecke1}
&&\lim_{t\to 0} \biggl[ \frac{-1}{(-2i)^{k-1}} \overline{\int^{\reg}_{[i\infty, \gamma(i\infty)]} \xi_{2-k}(h)(\tau)\left( \tau- \overline{\gamma\left( x+\frac{1}{c^2t}i\right)}\right)^{k-2}d\tau}\\
\nonumber &&-\frac{\chi^{-2k}_{\theta}(\gamma)\overline{b(0)}}{2^{k-1}(k-1)}\left( c\left(x+\frac{1}{c^2t}i\right)+d\right)^{2-k} \left(\frac{1}{c^2t} + \frac{1}{c} \right)^{k-1}\biggr],
\end{eqnarray}
where $\gamma = \sm a&b\\c&d\esm\in \Gamma_0(4NM^2)$ such that $r = \gamma(i\infty)$, and $c>0$ and $b(0)$ is the constant term of $\xi_{2-k}(h)$.
Since we have
\[
\xi_{2-k}(h)(z) = \xi_{2-k}(g)(M^2z),
\]
by Lemma \ref{Mtimes}, (\ref{Hecke1}) is equal to
\begin{eqnarray} \label{Hecke2}
\lim_{t\to 0} &&\biggl[ \frac{-1}{(-2i)^{k-1}} \overline{\int^{\reg}_{[i\infty, \gamma(i\infty)]} \xi_{2-k}(g)(M^2 \tau)\left( \tau- \overline{\gamma\left( x+\frac{1}{c^2t}i\right)}\right)^{k-2}d\tau}\\
\nonumber &&-\frac{\chi^{-2k}_{\theta}(\gamma)\overline{b(0)}}{2^{k-1}(k-1)}\left( c\left(x+\frac{1}{c^2t}i\right)+d\right)^{2-k} \left(\frac{1}{c^2t} + \frac{1}{c} \right)^{k-1}\biggr]\\
\nonumber =\lim_{t\to 0} &&\biggl[ \frac{-1}{M^{2k-2}(-2i)^{k-1}} \overline{\int^{\reg}_{[i\infty, \gamma'(i\infty)]} \xi_{2-k}(g)(\tau)\left( \tau- \overline{\gamma'\left( M^2 x+\frac{1}{(c')^2(M^2t)}i\right)}\right)^{k-2}d\tau}\\
\nonumber &&-\frac{\chi^{-2k}_{\theta}(\gamma)\overline{b(0)}}{M^{2k-2}2^{k-1}(k-1)}\left( c'\left(M^2x+\frac{1}{(c')^2(M^2t)}i\right)+d\right)^{2-k} \left(\frac{1}{(c')^2(M^2t)} + \frac{1}{c'} \right)^{k-1}\biggr],
\end{eqnarray}
where $\gamma' = \sm a& M^2b\\ c/M^2 & d\esm\in \Gamma_0(4NM^2)$ and $c' = c/M^2$.
Since $\gamma'(i\infty) = M^2r$, we see that (\ref{Hecke2}) is equal to $\frac{1}{M^{2k-2}} Q_g(M^2r)$.
\end{proof}

We are ready to prove Theorem \ref{determine}.

\begin{proof} [\bf Proof of Theorem \ref{determine}]
For $N\in\NN$, let
\[
C_N = \{ r\in\QQ\ |\ \text{$r$ is equivalent to $i\infty$ under the action of $\Gamma_0(4N)$}\}.
\]
Since $Q_{\hat{H}(f)}(r) = Q_{\hat{H}(g)}(r)$ for all $r\in C_N$, we have $Q_{\hat{H}(f-g)}(r) = 0$ for all $r\in C_N$.
We consider
\[
h(z) := \frac{1}{M^2} \hat{H}(f-g)(M^2z).
\]
Then, by Lemma \ref{Heckeoperator}, we have
\[
Q_h(r) = \frac{1}{M}Q_{\hat{H}(f-g)}(M^2r) =0
\]
for $r\in C_{NM^2}$.
Note that we have $Q_{h-\hat{H}(f-g)}(r) = 0$ for all $r\in C_{NM^2}$ and the constant term of $\xi_{2-k}(h-\hat{H}(f-g))$ vanishes.
If $\gamma'\in \Gamma_0(4NM^2)$, then, by Corollary \ref{limitperiod}, we obtain
\[
\lim_{t\to\infty}
\psi(h-\hat{H}(f-g))_{\gamma'}(\gamma(x+it)) = (Q_{h-\hat{H}(f-g)} - Q_{h-\hat{H}(f-g)}|_{2-k}\gamma')(r) = 0
\]
for $r\neq \gamma'^{-1}(i\infty)$.
Note that by Lemma \ref{periodfunction} we have
\[
\psi(h-\hat{H}(f-g))_{\gamma'}(z) = \frac{1}{(-2 i)^{k-1}} \overline{\int^{\reg}_{[\gamma'^{-1}(i\infty), i\infty]} \xi_{2-k}(h-\hat{H}(f-g))(\tau)(\tau-\overline{z})^{k-2}d\tau}
\]
for $z\in\HH$.
We denote $\gamma'$ by $\sm a'&b'\\c'&d'\esm$.
This regularized integral is well defined for $z = x+iy$ such that $y>-y_1$ and $\frac{y}{|c'z+d'|^2} > -\frac{y_1}{|c'z_1+d'|^2}$ for a fixed complex number $z_1 = x_1 + iy_1 \in \HH$ by Lemmas \ref{generalpoint} and \ref{welldefinedforcusp}.
Moreover, we can see that it is holomorphic at $z$ with  $y>-y_1$ and $\frac{y}{|c'z+d'|^2} > -\frac{y_1}{|c'z_1+d'|^2}$ by using the same argument as in the proof of Lemma \ref{coboundary}.
By the identity theorem, we have  $\psi(h-\hat{H}(f-g))_{\gamma'} = 0$ for all $\gamma'\in \Gamma_0(4NM^2)$.
This means that $(h-\hat{H}(f-g))^-$ is a harmonic weak Maass form in $H_{2-k}(\Gamma_0(4NM^2))$.
By Theorem \ref{nonexistence}, we have $(h-\hat{H}(f-g))^- = 0$, and hence $\xi_{2-k}(h-\hat{H}(f-g))=0$.
This implies that $\xi_{2-k}\hat{H}(f-g)$ does not have a principal part since the order of the pole of $\xi_{2-k}(h)$ is $M^2$ times the order of the pole of $\xi_{2-k}\hat{H}(f-g)$ in the Fourier expansion at the cusp $i\infty$.

Note that $f,g \in M^!_{0}(\Gamma_0^*(N))$ and $\Gamma_0^*(N)$ has only one cusp $i\infty$.
If $f-g$ has a Fourier expansion as in (\ref{Fourierinfinite}), then, by (\ref{shadowW}), we have
\[
\sum_{n<0}{a(dn)} = 0
\]
for every $d>0$ since $W(f-g) = -\frac12 \xi_{2-k}\hat{H}(f-g)$ does not have a principal part in the Fourier expansion at the cusp $i\infty$.
This implies that  the weakly holomorphic modular function
 $f-g$ does not have a pole at $i\infty$.
Therefore,  $f-g$ is a holomorphic function on a compact Riemann surface, so it is a constant function.
By assumption, the constant terms are zero in the Fourier expansions of $f$ and $g$ at the cusp $i\infty$.
Therefore, $f-g$ has a vanishing constant term, and hence $f-g = 0$.
This implies that
$\mathrm{Tr}_d(f) = \mathrm{Tr}_{d}(g)$  for each integer $d$.
\end{proof}

\section*{Acknowledgements}
The authors appreciate Roelof Bruggeman for his helpful comments which improved the exposition of this paper a lot.

 
\end{document}